\documentclass{article}

\PassOptionsToPackage{numbers,sort&compress}{natbib}
\usepackage[preprint, main]{neurips_2026}
\usepackage{enumitem}

\usepackage{amsmath,amssymb,amsthm}
\usepackage{bm}
\usepackage{algorithm}
\usepackage{algpseudocode}
\usepackage{graphicx}
\usepackage{booktabs}
\usepackage{hyperref}
\usepackage{cleveref}
\usepackage{subcaption}
\usepackage{xcolor}

\newtheorem{proposition}{Proposition}
\newtheorem{corollary}{Corollary}
\newtheorem{example}{Example}

\newtheorem{theorem}{Theorem}

\newtheorem{remark}{Remark}

\newcommand{\R}{\mathbb{R}}
\newcommand{\x}{\bm{x}}
\newcommand{\z}{\bm{z}}

\author{%
  \begin{tabular}{cc}
    \begin{tabular}{c}
      \href{https://orcid.org/0009-0008-3167-3992}{%
        \includegraphics[scale=0.06]{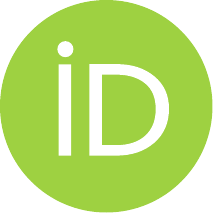}\hspace{1mm}Kevin Slote}\\
      \normalfont Clarkson Center for Complex Systems Science\\
      \normalfont Potsdam, NY 13676 \\
      \normalfont\texttt{kslote@clarkson.edu}
    \end{tabular}
    &
    \begin{tabular}{c}
      \href{https://orcid.org/0000-0001-7083-7592}{%
        \includegraphics[scale=0.06]{orcid.pdf}\hspace{1mm}Erik Bollt}\\
      \normalfont Clarkson Center for Complex Systems Science\\
      \normalfont Potsdam, NY 13676 \\
      \normalfont\texttt{embollt@clarkson.edu}
    \end{tabular}
    \\[2ex]
    \multicolumn{2}{c}{%
      \begin{tabular}{c}
        \href{https://orcid.org/0000-0002-1236-742X}{%
          \includegraphics[scale=0.06]{orcid.pdf}\hspace{1mm}Jeremie Fish}\\
        \normalfont Clarkson Center for Complex Systems Science\\
        \normalfont Potsdam, NY 13676 \\
        \normalfont\texttt{jafish@clarkson.edu}
      \end{tabular}
    }
  \end{tabular}
}

\title{Deep-Koopman-KANDy: Dictionary Discovery for Deep-Koopman Operators with Kolmogorov-Arnold Networks for Dynamics}
 
\begin{document}
\maketitle

\begin{abstract}
Symbolic library --- or Koopman dictionary --- selection is a fundamental challenge in data-driven dynamical systems. Extended Dynamic Mode Decomposition (EDMD), Sparse Identification of Nonlinear Dynamics (SINDy), and Kolmogorov--Arnold Networks for Dynamics (KANDy) all require the practitioner to commit to a function library at training time; Deep-Koopman Operators avoid this commitment but produce uninterpretable latent observables. We propose Deep-Koopman-KANDy, a structured approach to post-hoc symbolic dictionary readout that combines Deep-Koopman modeling with Kolmogorov-Arnold Networks for Dynamics (KANDy). The encoder and decoder of a Deep-Koopman Operator are replaced with two-layer Kolmogorov--Arnold Networks (KANs), and a level-set construction together with a chain-rule gradient identity exposes the compositional structure of the learned observables in a basis chosen \emph{after} training. We evaluate the method on the Lorenz system, the Chirikov standard map, the Ikeda map, and the Arnold cat map. On Lorenz it recovers the target dictionary $\{x,y,z,xy,xz\}$ with perfect recall and Jaccard score $0.79\pm0.06$; on the standard map it recovers a low-order Fourier basis matching the analytical structure; on Ikeda---which has no sparse polynomial representation---a misspecified polynomial readout still recovers the correct foliation coordinate $g\approx x^2+y^2$ together with a nontrivial outer function; and on the Arnold cat map --- used as a negative control because finite-dimensional Koopman closure is provably impossible --- the method fails to find a sparse closure, as expected.
\end{abstract}

\section*{Introduction}

Data-driven discovery of governing equations offers a compelling route to forecasting and controlling nonlinear dynamical systems, but current approaches require an unsatisfying up-front choice of candidate libraries. Methods such as Extended Dynamic Mode Decomposition (EDMD)~\cite{Williams2014, williams2016kernel}, Sparse Identification of Nonlinear Dynamics (SINDy)~\cite{sindy}, or Kolmogorov-Arnold Networks for Dynamics (KANDy)~\cite{kandy} depend on hand-crafted function libraries specified before fitting, while Deep-Koopman autoencoders avoid manual dictionary design at the cost of opaque latent observables~\cite{sindy_ae, lusch2018, li2017, otto2019}. For scientific model discovery, this is a central bottleneck. We address the timing of the library choice with Deep-Koopman-KANDy, which replaces standard encoder-decoder networks in Deep-Koopman Operators with Kolmogorov–Arnold Networks (KANs) -- allowing for post-training pruning of edge activations, a concrete advantage over multilayer perceptrons (MLPs) --- and recovers candidate structured, interpretable Koopman observables from the learned lifting that can then be read out symbolically in a user-chosen basis.
 
Given a dynamical system $\mathbf{x}_{t+1} = F(\mathbf{x}_t)$ on $\mathbb{R}^n$, the Koopman operator $\mathcal{K}$ acts on scalar observables $\varphi : \mathbb{R}^n \to \mathbb{R}$ by composition: $(\mathcal{K}\varphi)(\mathbf{x}) = \varphi(F(\mathbf{x}))$. When a finite set of observables is (approximately) closed under $\mathcal{K}$, the dynamics reduce to a finite-dimensional linear system in lifted coordinates, enabling classical techniques in numerical linear algebra for model reduction, prediction, data fusion, and control \cite{rowley2009spectral, budisic2012applied, brunton2022modern}, with applications in fluid dynamics \cite{mezic2013fluids,sharma2016correspondence}, energy modeling in buildings \cite{georgescu2015building} and oceanography \cite{giannakis2015spatiotemporal}, and molecular kinetics \cite{wu2017variational}.

\begin{figure}
    \centering
    \includegraphics[width=1.0\linewidth]{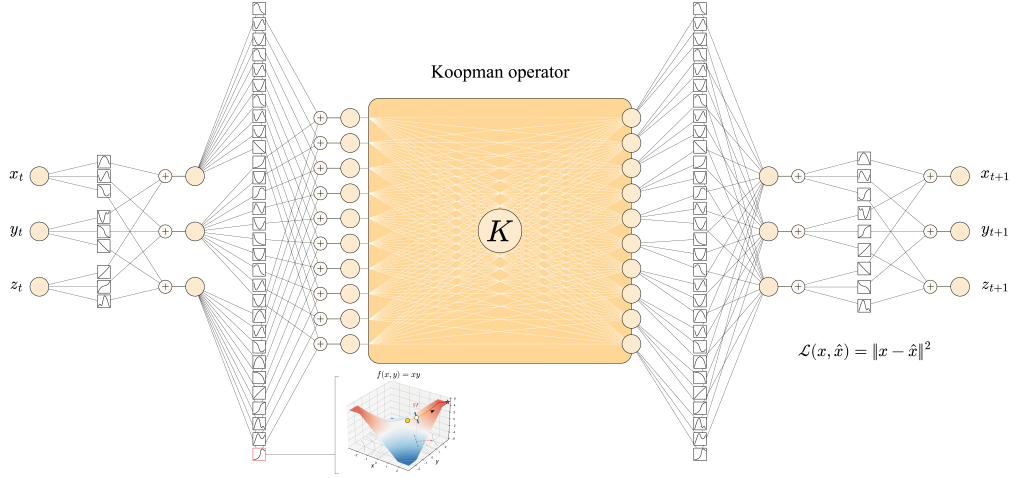}
    \caption{Deep-Koopman-KANDy architecture. A two-layer KAN encoder
    (\textbf{left}) lifts the state $\mathbf{x}$ into a structured latent
    observable $\mathbf{z}$. A stable linear generator $K$ (\textbf{center})
    advances the dynamics. A KAN decoder (\textbf{right}) inverts the lifting.
    Unlike a deep autoencoder, the KAN encoder exposes the compositional
    structure of each observable, allowing a level-set analysis of the
    higher-dimensional manifolds (\textbf{inset}) constructed from incident
    1D splines in the hidden layer.}
    \label{fig:network}
\end{figure}

\paragraph{Prior Work \& Literature Gap: Data-Driven Approximations and their Limitations.} 
Dynamic Mode Decomposition (DMD) and its extension EDMD provide practical finite-dimensional approximations of the Koopman operator~\cite{brunton2022modern, Williams2014, schmid2010dynamic}. EDMD lifts the state via a dictionary of nonlinear observables, improving upon standard DMD in many contexts; recent work further advances this by coupling EDMD with a trainable neural network dictionary~\cite{li2017}. Sparse methods such as SINDy~\cite{sindy} and KANDy~\cite{kandy} discover governing equations from practitioner-specified libraries of candidate functions. Deep-Koopman methods~\cite{lusch2018, li2017, otto2019} bypass dictionary selection by learning observables via neural-network autoencoders, but at the cost of interpretability—the latent features are unknown. A common limitation shared by all these approaches is the need for a function-class commitment: EDMD requires a handcrafted observable dictionary at training time, SINDy and KANDy require a predefined candidate library at training time, and deep autoencoders sacrifice interpretability for flexibility by deferring it indefinitely.

\paragraph{Kolmogorov--Arnold Networks.}
KANs~\cite{liu2025kan} replace each weight with a learnable univariate spline, exposing the functional form of the learned mapping at every edge. They are competitive with MLPs on regression, PDE solving, and scientific benchmarks~\cite{liu2025kan, coxkan2025, naturegraphkan2025}. KANDy~\cite{kandy} combines a single KAN layer with a SINDy-style identification framework. Because a single KAN layer represents only additive functions $g(\mathbf{x}) = \sum_i \psi_i(x_i)$, KANDy must lift cross terms (e.g.\ $xy$) into the dictionary by hand. Other recent work applies KANs to discrete chaotic maps and ordinary differential equations~\cite{Panahi2025KANModelDiscovery, Bagrow2025, Koenig2024,
Koenig2025}.

\paragraph{This work.}
Deep-Koopman-KANDy replaces the encoder and decoder of a Deep-Koopman architecture with two-layer KANs (Figure~\ref{fig:network}) and applies a level-set analysis on multilinear activation manifolds to recover compositional structure (Figure~\ref{fig:network}, inset). Two layers are required because the first KAN layer can only produce additive features; the second supplies the compositional depth needed for cross terms. The architecture decouples representation from description: the encoder learns a basis-free foliation, and the level-set machinery describes that foliation in any chosen basis after training. A wrong basis is then a descriptive failure---a denser readout of the same learned object---rather than a representational failure that requires retraining.

\paragraph{Contributions.}
\begin{itemize}[leftmargin=1.2em, itemsep=0pt, topsep=2pt]
\item A two-layer KAN encoder--decoder architecture for Deep-Koopman models with
a stably parameterized linear generator.
\item A level-set analysis of activation manifolds that recovers the
compositional decomposition $f = h \circ g$ of each learned observable.
\item Empirical demonstrations on continuous and discrete systems, including a negative-control case (Arnold cat map) where finite-dimensional spectral closure is provably impossible.
\item Evidence that misspecified readout dictionaries (e.g.\ a polynomial readout on Ikeda) degrade gracefully---$g$ becomes denser, $h$ absorbs the residual---rather than failing catastrophically.
\end{itemize}
The aim is not to replace SINDy, KANDy, or EDMD downstream, but to learn the
candidate observable structure that informs the library those methods use.

\section*{Background}
\label{sec:background}

\begin{figure}[!htp]
    \centering
    \includegraphics[width=1.0\linewidth]{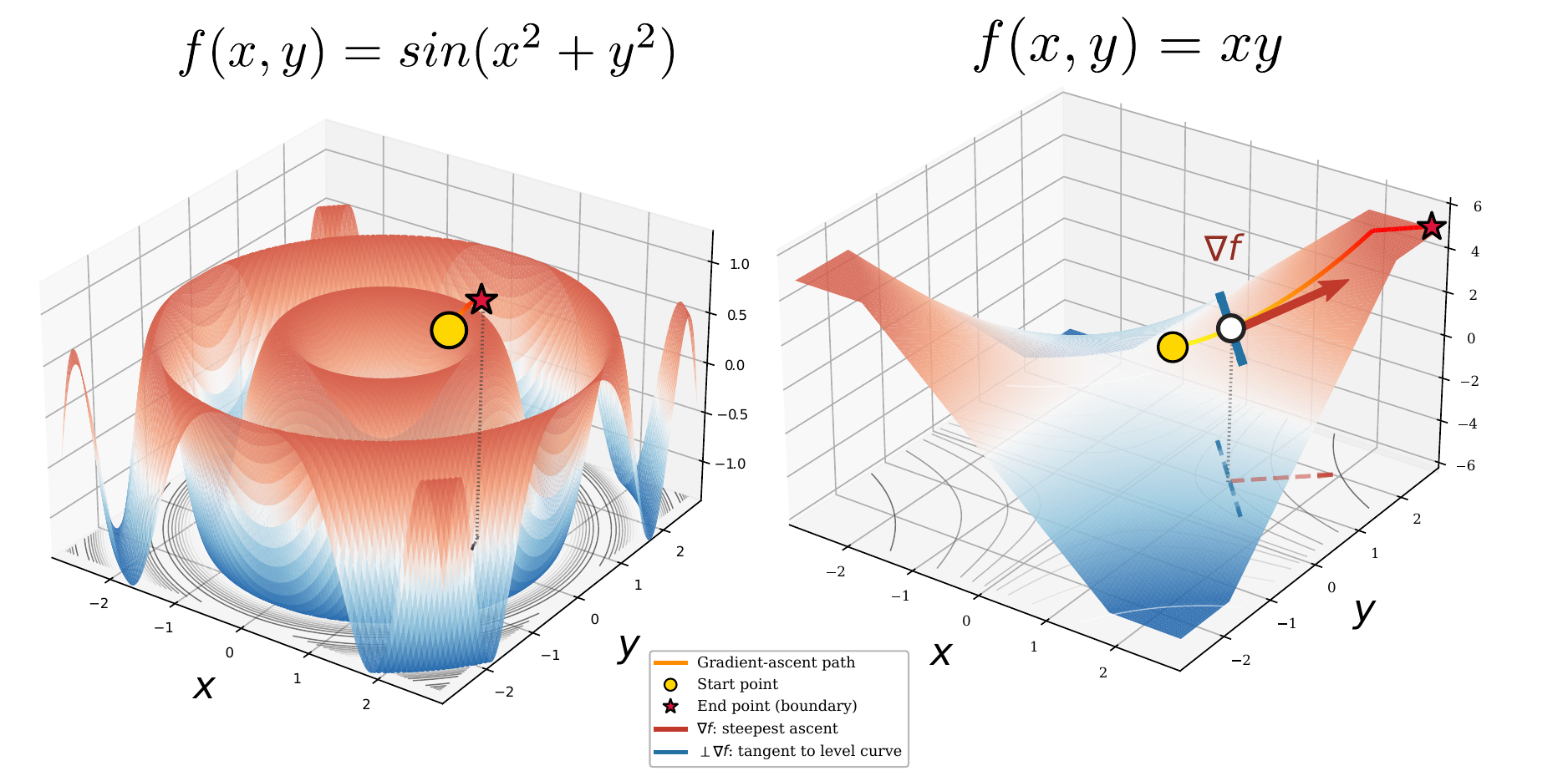}
    \caption{The hidden activation layer of the KAN defines a manifold in latent
    space. The level sets of the learned coordinate functions induce a foliation
    that we exploit to recover compositional structure.}
    \label{fig:manifold}
\end{figure}

\paragraph{Koopman operators.}
For a continuous-time system $\dot{\x}=f(\x)$, the Koopman generator
$\mathcal{A}$ acts on observables by
\begin{equation}
  \mathcal{A}\varphi = \nabla\varphi \cdot f(\x),
  \label{eq:koopman-generator}
\end{equation}
so eigenfunctions $\varphi_\lambda$ with $\mathcal{A}\varphi_\lambda =
\lambda\varphi_\lambda$ evolve as $\varphi_\lambda(\Phi^t(\x))=e^{\lambda
t}\varphi_\lambda(\x)$, where $\Phi^t$ is the flow of $f$. A Deep-Koopman
model~\cite{lusch2018} learns an encoder $\mathcal{E}\colon\R^n\to\R^N$, a
decoder $\mathcal{D}\colon\R^N\to\R^n$, and a matrix $K\in\R^{N\times N}$ such
that
\begin{equation}
  \z_t = \mathcal{E}(\x_t),\qquad
  \z_{t+\Delta t} \approx e^{K\Delta t}\z_t,\qquad
  \x_{t+\Delta t} \approx \mathcal{D}(\z_{t+\Delta t}).
  \label{eq:deep-koopman}
\end{equation}
We parameterize the generator as $K = \Omega - L^\top L$ with $\Omega$
skew-symmetric and $L$ lower triangular. This forces
$\mathrm{Re}(\sigma(K))\le 0$ and hence
$|\lambda(\exp(K\Delta t))|\le 1$, ensuring stable rollouts without spectral
penalties or projection steps.

\begin{theorem}[Kolmogorov--Arnold, 1957]
Every continuous $f:[0,1]^n\to\R$ admits the representation
\begin{equation}
  f(\x) = \sum_{q=0}^{2n}\Phi_q\!\left(\sum_{p=1}^n \psi_{q,p}(x_p)\right),
  \label{eq:KAT}
\end{equation}
with continuous univariate $\psi_{q,p},\Phi_q$.
\end{theorem}

\paragraph{KANDy.}
KANDy~\cite{kandy} uses a hand-crafted SINDy-style library or a learned
Koopman-lifted dictionary $\Theta$. A 1-layer KAN feature
$\hat{f_\theta}(\Theta) = \mathrm{KAN}_\theta(\Theta) \approx d\x/dt$ together
with $\Theta(\x) = [\theta_1(\x),\dots,\theta_N(\x)]$ supports a sparse
dynamical model. Single-layer KANs cannot emit multiplicative cross terms
without an algebraic workaround: producing $xy$ requires expanding
$(x+y)^2 = x^2 + 2xy + y^2$ and cancelling the squares, which is the
``quadratic obstruction'' identified in~\cite{kandy}. Two layers resolve this:
the inner $\{\psi_{q,p}\}$ provide additive features and the outer $\{\Phi_q\}$
provide the cross-term compositions.

\section*{Methods}

\paragraph{Architecture.}
The encoder maps $\x\in\R^n$ to $\z\in\R^d$ in two KAN layers, with $m$
intermediate units. We denote the architecture $[n,m,d]$. The first layer
computes
\begin{equation}
  u_j = \sum_{i=1}^{n}\phi^{(1)}_{j,i}(x_i),\qquad j=1,\dots,m,
  \label{eq:layer1}
\end{equation}
where each edge activation is
\begin{equation}
  \phi_{j,i}(x_i) = c_{j,i}\,\mathrm{SiLU}(x_i)
    + \sum_{k=1}^{G+S} w_{j,i,k}\,B_k(x_i),
  \label{eq:kanlinear}
\end{equation}
with $\{B_k\}$ a set of order-$S$ B-splines on a grid of size $G$, and
$\mathrm{SiLU}(x) = x\sigma(x)$ a smooth residual. The B-spline basis is the
default; Chebyshev polynomials, Fourier features, or radial basis functions are
drop-in replacements. The second layer maps $\R^m\to\R^d$:
\begin{equation}
  z_k = \sum_{j=1}^{m}\phi^{(2)}_{k,j}(u_j)
      = \sum_j \Phi_{k,j}\!\left(\sum_i \psi_{j,i}(x_i)\right),\qquad
      k=1,\dots,d.
  \label{eq:layer2}
\end{equation}
By the Kolmogorov--Arnold theorem each $z_k$ approximates a broad class of
continuous functions of $(x_1,\dots,x_n)$. The width $m$ controls compositional
richness; $d$ controls dictionary size.

\paragraph{Stable propagator.}
The Koopman generator is parameterized as $K = \Omega - L^\top L$ with $\Omega$
skew-symmetric and $L^\top L\succeq 0$, guaranteeing $|\lambda|\le 1$ for all
discrete-time eigenvalues. One-step propagation uses
\begin{equation}
  \z_{t+\Delta t} = \exp(K\,\Delta t)\,\z_t.
  \label{eq:koopman-prop}
\end{equation}

\paragraph{Training objective.}
For consecutive pairs $(\x_t,\x_{t+1})$ the loss is
\begin{equation}
  \mathcal{L} = \gamma\,\|\hat{\x}_{t+1}-\x_{t+1}\|^2
              + \alpha\,\|\hat{\z}_{t+1}-\mathcal{E}(\x_{t+1})\|^2
              + \eta\,\|\hat{\x}_t-\x_t\|^2,
  \label{eq:loss}
\end{equation}
with $\hat{\z}_{t+1} = \exp(K\Delta t)\z_t$,
$\hat{\x}_{t+1}=\mathcal{D}(\hat{\z}_{t+1})$, and
$\hat{\x}_t = \mathcal{D}(\mathcal{E}(\x_t))$. The three terms enforce
prediction accuracy ($\gamma$), latent consistency ($\alpha$), and autoencoder
fidelity ($\eta$). We use $\gamma=\alpha=1$, $\eta=0.5$ throughout, which
downweights reconstruction so the encoder prioritizes predictive coordinates.

\paragraph{Pruning.}
After training, edges with normalized importance below a threshold $\tau$ are
removed from the encoder and decoder, and the model is retrained briefly at a
reduced learning rate. This prunes spurious spline contributions and
substantially improves dictionary parsimony in the level-set decomposition.
We choose $\tau=0.03$ from a sweep on Lorenz (Appendix
Table~\ref{tab:prune-sweep}); higher thresholds degrade accuracy, lower
thresholds leave too many false positives.

\paragraph{Level-set decomposition.}
\label{sec:levelset}
Each latent coordinate is a scalar function $z_k = f(\x)$. We decompose it as
$f \approx h\!\circ\!g$, where $g$ is a sparse inner function in a chosen
post-hoc basis and $h$ is a univariate outer function.

The inner $g$ is recovered by Lasso on a design matrix $\Theta\in\R^{N\times P}$
with columns drawn from any chosen basis (monomials of total degree $\le D$
in our experiments, but trigonometric or RBF columns are equally admissible):
\begin{equation}
  \min_{\bm{a}}\;\frac{1}{2N}\|\Theta_s\bm{a}-f_s\|_2^2 + \lambda\|\bm{a}\|_1.
  \label{eq:lasso}
\end{equation}
Standardization (subscript $s$) ensures the regularizer treats columns
uniformly.

\begin{proposition}[Outer derivative formula]\label{prop:hprime}
If $f(\x)=h(g(\x))$ and $\nabla g(\x)\ne\bm{0}$, then
\begin{equation}
  h'\!\bigl(g(\x)\bigr)
    = \frac{\nabla f(\x)\cdot\nabla g(\x)}{\|\nabla g(\x)\|^2}.
  \label{eq:hprime}
\end{equation}
\end{proposition}
\begin{proof}
See Appendix~\ref{app:proof-hprime}.
\end{proof}

The right-hand side of \eqref{eq:hprime} is computed pointwise from samples and
binned over $g$ to recover $h$ up to integration constants (complete algorithm found in Appendix Section~\ref{sec:algorithm}).

Proposition~\ref{prop:hprime} recovers $h$ assuming $f = h\circ g$, but does
this factorization hold for the latent observables produced by training?
The relevant condition is that $\nabla f$ and $\nabla g$ be parallel at each
point---equivalent to $f$ being constant on level sets of $g$. In the ambient
space $\mathbb{R}^n$ this is restrictive: the polynomial identity
$xy = \tfrac{1}{4}(x+y)^2 - \tfrac{1}{4}(x-y)^2$ shows that $f = xy$ and
$g = (x+y)^2$ have transverse level sets, so no ambient $h$ exists. 

The dynamics, however, live on a $d$-dimensional attractor
$\mathcal{A}\subset\mathbb{R}^n$, and only the \emph{intrinsic} gradients
$\nabla^{\!\mathcal{A}} f := P_x\nabla f$ and
$\nabla^{\!\mathcal{A}} g := P_x\nabla g$---the ambient gradients projected
onto the tangent space $T_x\mathcal{A}$ via $P_x$---are dynamically
meaningful. Parallelism on $\mathcal{A}$ is much weaker than parallelism
in $\mathbb{R}^n$:

\begin{theorem}[Codimension reduction; informal]\label{thm:codim-informal}
On a $d$-dimensional attractor in $\mathbb{R}^n$, the pointwise condition
that $\nabla^{\!\mathcal{A}} f$ and $\nabla^{\!\mathcal{A}} g$ be parallel
has codimension $d-1$, versus codimension $n-1$ in the ambient space.
When $d=1$ it is automatic; in general, the obstruction is reduced by
$n-d = \mathrm{codim}\,\mathcal{A}$.
\end{theorem}

A formal statement and proof of Theorem~\ref{thm:codim-informal} appear as Theorem~\ref{thm:codim} in Appendix~\ref{app:proof-hprime}; Example~\ref{ex:cross} below illustrates the resolution for cross terms.
A quantitative companion (Theorem~\ref{thm:residual}) bounds the residual of the ambient chain-rule estimator by $\beta = \|\nabla^{\!\perp} g\|^2/\|\nabla g\|^2$, the squared sine of the angle between $\nabla g$ and $T_x\mathcal{A}$. 
Under a uniform-orientation model, the expected residual $\mathbb{E}[\beta]=(n-d)/n$ predicts false-positive coefficient magnitudes consistent with observations across all four systems (Appendix Table~\ref{tab:predicted-residual}).

\begin{example}[The cross-term obstruction on a curve]
\label{ex:cross}
Take $f(x,y)=xy$ and $g(x,y)=(x+y)^2$ on $\mathbb{R}^2$. Ambient
gradients are $\nabla f=(y,x)$, $\nabla g = 2(x+y)(1,1)$. For
$x\neq y$, these are not parallel, so no ambient $h$ exists with
$f = h\circ g$. Now restrict to a 1-dimensional attractor
$\mathcal{A}\subset\mathbb{R}^2$ with tangent $\tau(s)=(\cos\theta(s),
\sin\theta(s))$. The intrinsic gradients are scalars
$df_x(\tau)=y\cos\theta + x\sin\theta$ and
$dg_x(\tau)=2(x+y)(\cos\theta+\sin\theta)$. By Theorem~\ref{thm:codim}(1),
the intrinsic factorization holds whenever both are nonzero, and
$$
  h'\bigl(g(x,y)\bigr) =
    \frac{y\cos\theta + x\sin\theta}{2(x+y)(\cos\theta+\sin\theta)}.
$$
The right-hand side is constant on level sets of $g$ restricted to
$\mathcal{A}$ (as required for $h$ to exist). The would-be obstruction
$\nabla f - \lambda\nabla g$ for $\lambda = h'(g)$ is nonzero in
$\mathbb{R}^2$ but lies in $N_x\mathcal{A}$ pointwise.
\end{example}


\begin{figure*}[!htp]
    \includegraphics[width=\textwidth]{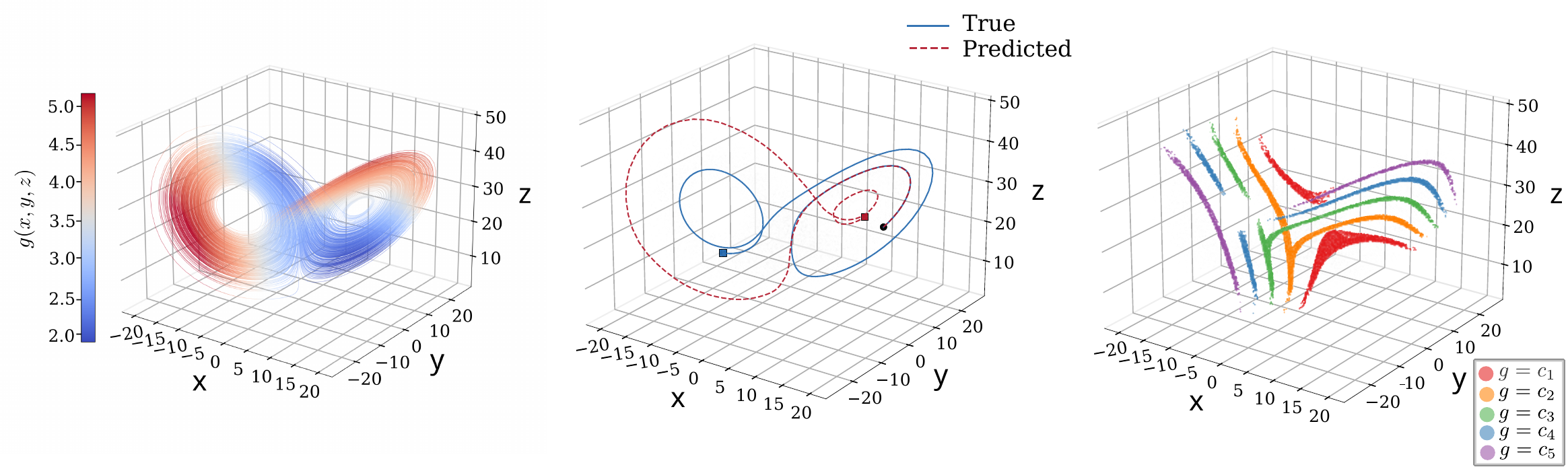}
    \caption{\textbf{Lorenz attractor and learned Koopman observable.}
    \textbf{(left)}~Attractor colored by the learned observable $g(x,y,z)$.
    \textbf{(center)}~Ground-truth and predicted trajectory over
    $\approx 2\tau_\Lambda$. \textbf{(right)}~Level-set bands $g(x,y,z)\approx
    c_i$ induce a structured foliation of the attractor.}
    \label{fig:lorenz_attractor}
\end{figure*}

\section*{Results}
\label{sec:results}

\subsection*{Lorenz system}

\begin{equation}
  \dot{x} = \sigma(y-x),\quad
  \dot{y} = x(\rho-z)-y,\quad
  \dot{z} = xy-\beta z,
  \label{eq:lorenz}
\end{equation}
with $\sigma=10$, $\rho=28$, $\beta=8/3$. The vector field is supported by the
target dictionary $\mathcal{T} = \{x,y,z,xy,xz\}$.

\textbf{Headline result.} Deep-Koopman-KANDy with architecture $[3,5,5]$
recovers $\mathcal{T}$ exactly at the union level: every term in $\mathcal{T}$
appears in at least one latent coordinate across all five seeds, with
per-seed Jaccard $0.79\pm 0.06$ and union recall $1.0$ (Table~\ref{tab:ablation-main}).
Each $z_k$ contains a subset of $\mathcal{T}$ together with residual false
positives $\{x^2,y^2,z^2\}$ at coefficients below $10^{-2}$, attributable to
polynomial-approximation residuals of the learned splines. All five latent
coordinates achieve median $R^2(h\circ g) \ge 0.92$
(Appendix Table~\ref{tab:levelset-decomposition}).

\textbf{Spectral interpretation.} The $z$-coordinate aligns with the dominant
spiral pair at $\pm 1.43$\,Hz ($r=0.94$, $R^2=0.99$). The cross terms $xy$ and
$xz$ are distributed across all five Koopman modes ($R^2=0.85$ and $0.77$),
indicating that cross-term dynamics arise from collective interference rather
than a single mode---a clean reading of mode mixing that opaque latent models
cannot provide.

Training details (60 trajectories, 6000 RK4 steps each, AdamW with learning
rate $10^{-3}$, 200 epochs followed by 100 epochs at $5\times 10^{-4}$ after
pruning at $\tau=0.03$) are listed in Appendix Table~\ref{tab:prune-sweep}. The pruning
step reduces the manifold dimensionality (the spline $\psi_{5,5}(z)$ vanishes)
and lowers false-positive rates from $\sim 4$ to $2.2$ per dimension while
leaving one-step MSE essentially unchanged ($2.6\to 2.5\times 10^{-3}$).
A degree-3 Lasso readout with $\lambda=10^{-6}$ then yields $\mathcal{T}$
exactly.\footnote{\href{https://anonymous.4open.science/r/deep-koopman-KANDy-D966}{Code,
ablation sweeps, and full model-size scans on GitHub.}}

\subsection*{Standard map}

\begin{figure}[!htpb]
    \centering
    \includegraphics[width=1.0\linewidth]{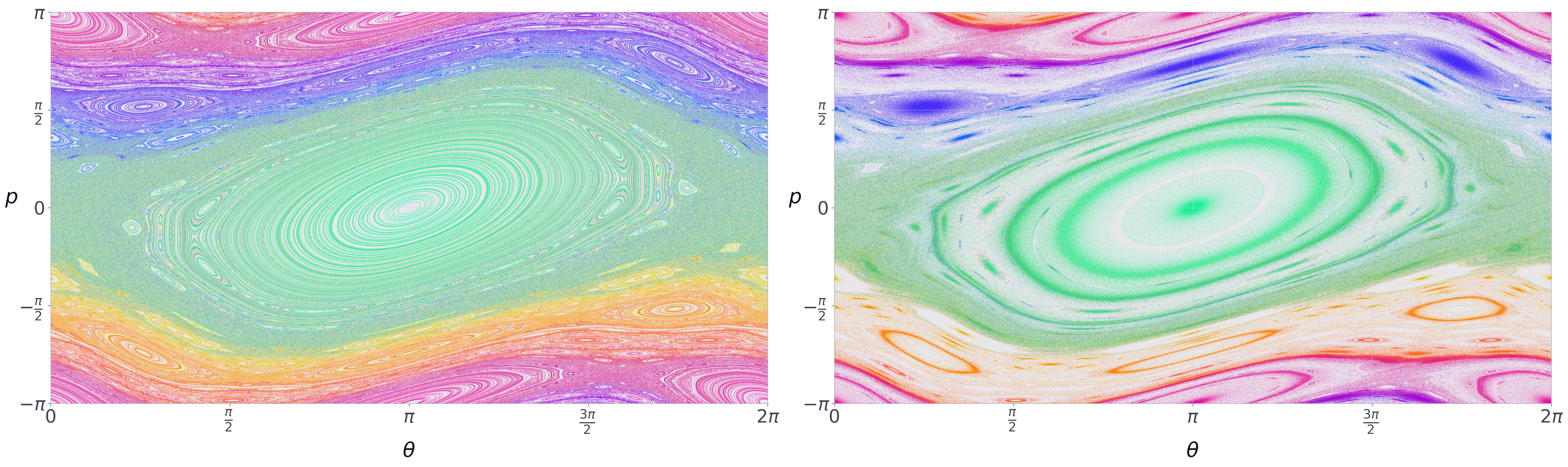}
    \caption{Deep-Koopman-KANDy on the standard map at $\kappa=0.97$.
    \textbf{Left:} ground-truth phase portrait. \textbf{Right:} predicted phase
    portrait as the encoder width $m$ increases. Larger $m$ reproduces both
    the KAM tori (invariant curves of regular motion) and the surrounding
    chaotic sea.}
    \label{fig:standard_map}
\end{figure}

We evaluate the standard map
$(\theta_{n+1},p_{n+1}) = (\theta_n+p_n,\,p_n+\kappa\sin\theta_n)\bmod 2\pi$ at
$\kappa=0.97$, the mixed regime where regular tori coexist with chaos.

\textbf{The encoder recovers the analytical Fourier basis.} The dominant learned
observables span
\begin{equation}
\left\{
\begin{aligned}
&\cos\theta,\,\sin\theta,\,\cos p,\,\sin p,\\
&\cos\theta\cos p,\,\sin\theta\sin p,
\cos\theta\sin p,\,\sin\theta\cos p,\\
&\cos\theta\sin\theta,\,\cos p\sin p
\end{aligned}
\right\},
\label{eq:standard_map_dictionary}
\end{equation}
which in Fourier language is exactly $e^{i(m\theta+np)}$ with $|m|+|n|\le 2$,
the analytical observable structure for the standard map. The recovery is
independent of the propagator parameterization: the encoder learns a
candidate basis, and the propagator separately learns its evolution.

\textbf{The stable parameterization preserves geometry.} Of $128$ latent
dimensions, $\sim 77$ have $\sigma_i\approx 1$ and the rest lie in $[0.84,
0.95]$, allocating one subset to nearly conservative dynamics and the other
to mildly contractive dynamics consistent with mixing in the chaotic sea.
Long-horizon rollouts preserve KAM island geometry while populating the
chaotic sea.

\subsection*{Ikeda map (misspecified readout)}

The Ikeda map depends on $r^2 = x^2+y^2$ through trigonometric functions of
$t(r^2) = 0.4 - 6/(1+r^2)$. A degree-3 polynomial Lasso readout is therefore
misspecified by construction. We use it anyway, treating the polynomial as a
flexible interpolant rather than a structural prior.

\textbf{Misspecification degrades parsimony, not expressivity.} Across all four
latent coordinates of a $[2,4,4]$ encoder, the level-set decomposition holds
with $R^2(h\circ g) \in [0.92, 0.99]$ (Fig.~\ref{fig:iked_minimal},
Table~\ref{tab:ikeda-levelset-decomposition}). The recovered $g$ is not sparse:
each dimension picks up several higher-order monomials as the polynomial basis
strains to approximate the non-polynomial composition. The level sets
nonetheless form coherent bands across the attractor, and the gradient identity
recovers a nontrivial $h'(g)$ rather than collapsing to an affine correction.
This is clearest for $z_2$, whose binned derivative medians exhibit the
two-peak structure tracked by the fitted outer function.

\textbf{The compositional structure is recovered.} The level-set Lasso identifies $g \approx x^2+y^2$ from matched quadratic coefficients, and the gradient identity reveals an outer function well-described by $\{\cos(t(g)),\sin(t(g))\}$ with $t(g) = 0.4 - 6/(1+g)$. The foliation coordinate is correct; the outer map must still be expressed in the trigonometric basis to obtain a closed form. This is the central decoupling claim: when the readout dictionary is wrong, the error appears as a denser $g$ and a more complex $h$, not as a representational failure---an outcome that SINDy- and KANDy-style pipelines, which commit to the dictionary at training time, cannot achieve.

\subsection*{Arnold cat map (continuous-spectrum control)}

\begin{figure}
    \centering
    \includegraphics[width=1.0\linewidth]{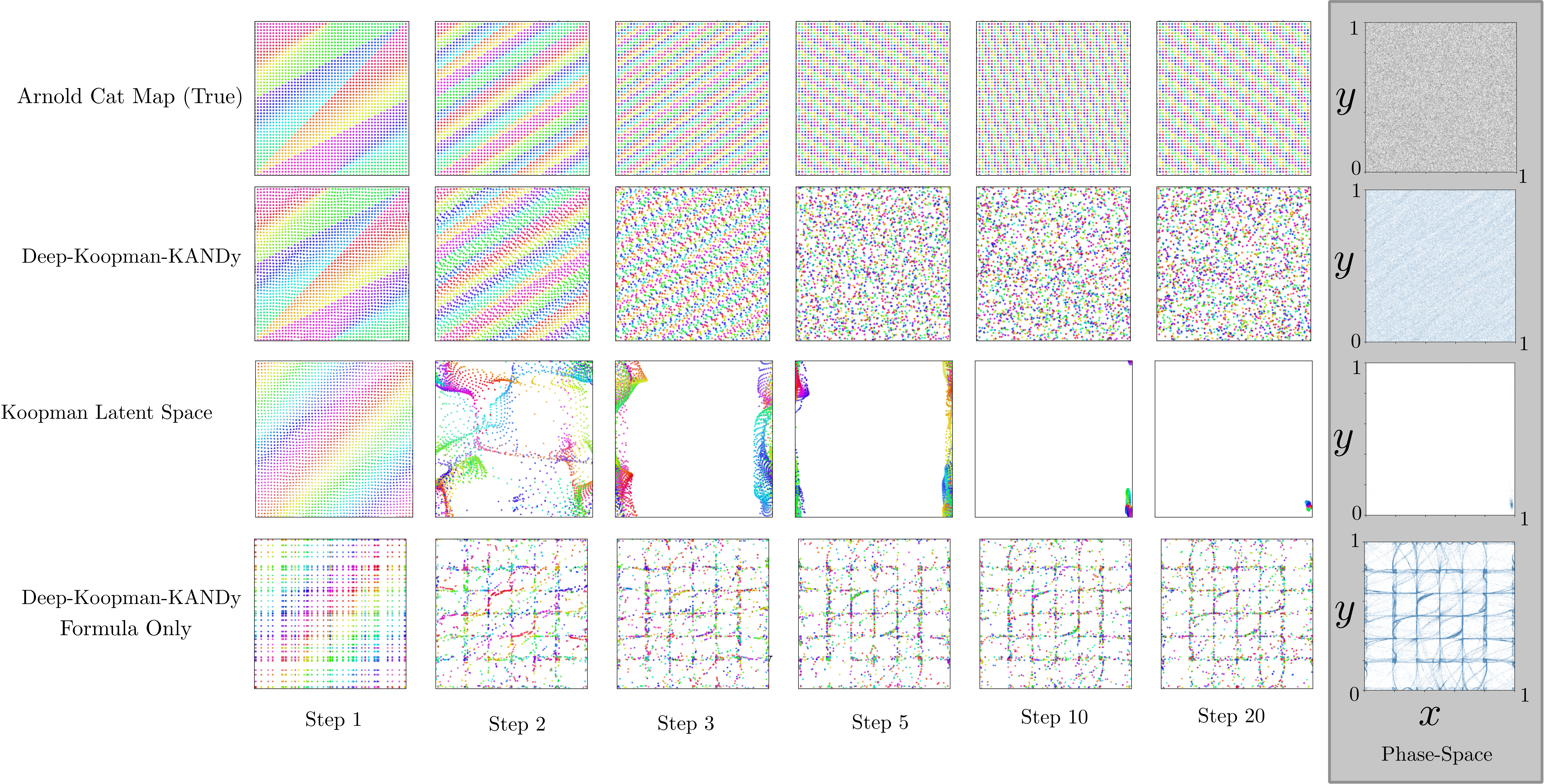}
    \caption{Arnold cat map: ground truth, full model, and two ablations.
    \textbf{Row 1:} ground-truth iterations and phase portrait.
    \textbf{Row 2:} autoregressive rollout of the full Deep-Koopman-KANDy model
    and its phase portrait. \textbf{Row 3:} latent-only rollout ($M^n$ after a
    single encode); amplitudes decay and the rollout converges to the ergodic
    mean. \textbf{Row 4:} symbolic-only rollout from the recovered Fourier
    formula ($D\circ M\circ E$ with no re-encoding); a lattice phase portrait
    emerges from Gibbs artifacts at the sawtooth discontinuity. The full model
    rolls out beyond the Lyapunov time only when the latent generator and the
    KAN encoder/decoder act together.}
    \label{fig:cat_map_rollout}
\end{figure}

The Arnold cat map on the two-torus,
\begin{equation}
T:\begin{pmatrix}x\\y\end{pmatrix}
\mapsto
\begin{pmatrix}2&1\\1&1\end{pmatrix}
\begin{pmatrix}x\\y\end{pmatrix}\bmod 1,\qquad (x,y)\in[0,1)^2,
\label{eq:cat_map}
\end{equation}
is the canonical uniformly hyperbolic, area-preserving map with purely
continuous Koopman spectrum: the constant function is the only eigenfunction.
The Lyapunov exponent $\lambda_{\max} = \log((3+\sqrt5)/2) \approx 0.962$ sets
the predictability horizon at $\tau_\Lambda \approx 1.04$ iterates. We use this
system as a negative control: a finite-dimensional Koopman model should fail, and we ask whether ours fails
honestly.

We embed $(x,y)$ as $(\cos 2\pi x,\sin 2\pi x,\cos 2\pi y,\sin 2\pi y)\in\R^4$
to obtain a smooth torus representation, train a $[4,24,1024]$ encoder for
$200$ epochs on $200$ trajectories of length $500$, and reach train and
validation losses of $6\times 10^{-5}$.

\textbf{The spectral failure.} The learned $1024\times
1024$ propagator has $|\lambda|\in[0,0.958]$ with mean $\approx 0.85$; no
eigenvalue satisfies $|\lambda|>0.96$. This contrasts with Lorenz, where many
modes lie near $|\lambda|=1$, and is the spectral signature of the absent
eigenfunction structure. As model size grows
(Table~\ref{tab:cat_model_scaling}), one-step accuracy improves by two orders
of magnitude, but the prediction horizon extends by only $\sim 2$ steps,
consistent with the Lyapunov barrier. One-step prediction is sharp
($\mathrm{MSE}_1\approx 3\times 10^{-5}$, roughly $3{,}300\times$ below
persistence), but error grows exponentially and reaches the persistence scale
by step $\sim 7$ (Appendix Table~\ref{tab:cat_prediction}).

\textbf{Three rollout regimes expose the architecture's roles.} (i)~Pure latent
rollout ($M^n$ after a single encode) collapses to a point as latent
amplitudes decay like $0.958^n$, with information effectively gone after
$\sim 62$ steps. (ii)~Symbolic rollout of the recovered Fourier formula
$f_x \approx 1/2 - \sum_{k=1}^{6}\sin(2\pi k(2x+y))/(\pi k)$ and
$f_y \approx 1/2 - \sum_{k=1}^{6}\sin(2\pi k(x+y))/(\pi k)$ matches the
analytical $1/(\pi k)$ coefficients to within $0.2\%$, but iteration produces a
lattice phase portrait from Gibbs artifacts at the sawtooth discontinuity.
(iii)~Full autoregressive rollout ($D\circ M\circ E$, re-encoding at every
step) broadly fills the torus and remains mixing-like over long horizons,
because the encoder rebuilds nonlinear information at each step rather than
propagating a fixed latent basis.

\subsection*{Ablation study}

We compare Deep-Koopman-KANDy on Lorenz against external baselines and
controlled architectural ablations: SINDy-Autoencoder (SINDy-AE), EDMD with
dictionary learning (EDMD-DL), an MLP autoencoder with stable propagator
(MLP+stable), and ablations isolating KAN depth, encoder/decoder placement,
and the stability constraint. Dictionary recovery is measured against
$\mathcal{T}=\{x,y,z,xy,xz\}$ via Jaccard similarity and recall; for SINDy-AE
and EDMD-DL we apply symbolic regression to approximate a symbolic readout.

\begin{table}[t]
\centering
\small
\setlength{\tabcolsep}{2.5pt}
\renewcommand{\arraystretch}{1.12}
\begin{tabular}{@{}lrccccc@{}}
\toprule
\textbf{Variant}
& \textbf{Params} $\downarrow$
& \textbf{Wall} $\downarrow$
& \textbf{MSE} $\downarrow$
& \textbf{NRMSE}$_{2\tau}$ $\downarrow$
& \textbf{Jaccard} $\uparrow$
& \textbf{Rec.} $\uparrow$ \\
\midrule
\multicolumn{7}{@{}l}{\textit{External baselines}} \\
EDMD-DL                       & $2{,}501$ & $3.9$h  & $2.78\times10^{-1}$ & $\mathbf{0.91 \pm 0.07}$ & $0.60 \pm 0.00$ & $0.60$ \\
SINDy-AE                      & $5{,}280$ & $11.1$h & $\mathbf{1.93\times10^{-3}}$ & $1.01 \pm 0.07$ & $0.60 \pm 0.00$ & $0.60$ \\
\addlinespace[2pt]
\multicolumn{7}{@{}l}{\textit{Architectural ablations}} \\
1L-KAN + stable               & $300$     & $32$m   & $3.28\times10^{-3}$ & $1.16 \pm 0.22$ & $0.61 \pm 0.13$ & $0.84$ \\
KAN-enc + MLP-dec             & $448$     & $33$m   & $7.63\times10^{-3}$ & $1.27 \pm 0.21$ & $0.70 \pm 0.12$ & $0.96$ \\
MLP-enc + KAN-dec             & $450$     & $36$m   & $1.25\times10^{-2}$ & $1.39 \pm 0.20$ & $0.66 \pm 0.07$ & $0.92$ \\
MLP + stable                  & $98$      & $24$m   & $3.21\times10^{-2}$ & $1.18 \pm 0.25$ & $0.62 \pm 0.13$ & $0.88$ \\
2L-KAN + unconstrained        & $800$     & $45$m   & $3.63\times10^{-3}$ & $1.35 \pm 0.29$ & $0.63 \pm 0.05$ & $\mathbf{1.00}$ \\
\addlinespace[2pt]
\multicolumn{7}{@{}l}{\textit{Ours}} \\
Deep-Koopman-KANDy ($\ell_1$) & $800$     & $46$m   & $4.61\times10^{-3}$ & $1.10 \pm 0.35$ & $0.70 \pm 0.08$ & $\mathbf{1.00}$ \\
\textbf{Deep-Koopman-KANDy (pruned)} & $800$ & $3.7$h & $2.29\times10^{-3}$ & $1.13 \pm 0.21$ & $\mathbf{0.79 \pm 0.06}$ & $\mathbf{1.00}$ \\
\bottomrule
\end{tabular}
\caption{Dictionary recovery and one-step prediction on Lorenz over $5$ seeds.
Best results in \textbf{bold}.}
\label{tab:ablation-main}
\end{table}

The pruned Deep-Koopman-KANDy obtains the best Jaccard ($0.79\pm 0.06$) at
perfect recall. EDMD-DL and SINDy-AE recover only $\{x,y,z\}$ and miss both
cross terms (Jaccard $0.60$): EDMD-DL includes states by construction, and
SINDy-AE places nonlinear structure in the coefficient matrix rather than the
lifting map. Architectural ablations cut recall to $0.84$--$0.96$, showing
that two stacked KAN layers materially improve cross-term representation.
Removing the stability constraint preserves recall but increases false
positives (Jaccard $0.63$ vs.\ $0.70$); spectral constraints improve parsimony.
The pruning threshold $\tau=0.03$ removes about $18\%$ of the $40$ encoder
edges, restores MSE within $2\%$ of baseline after $100$ epochs of retraining,
and avoids the degradation seen at $\tau=0.05$ ($1.07\times$ MSE)
(Appendix Table~\ref{tab:prune-sweep}). Across $12$ runs, $\tau=0.03$ sits near
the $14$th percentile of edge scores, indicating that pruning targets
consistently low-importance connections rather than seed-specific artifacts.

\section*{Discussion}
Deep-Koopman-KANDy decouples representation from description --- the encoder represents the dynamics by learning a foliation of state space, and the readout describes that foliation in a basis chosen after training. The two-layer KAN encoder commits to no basis at training time; the post-hoc level-set analysis names the foliation in whichever basis is supplied. A wrong basis in classical methods is a representational failure --- the dynamics cannot be expressed --- whereas here it is a descriptive failure, yielding a denser readout of the same learned object. The Ikeda result makes this concrete: a polynomial readout is provably wrong for Ikeda, yet the level-set decomposition recovers compositional structure with $R^2 > 0.92$.

\paragraph{Two failure modes, two diagnostics.}
The decoupling shifts dictionary dependence from the lifting, where miscommitment is representational, to the symbolic readout, where any flexible basis can be tried without retraining. Two distinct failure modes arise. When the readout dictionary is misspecified but the dynamics admit some finite-dimensional structure (Ikeda), the decomposition degrades gracefully: $g$ becomes denser and $h$ absorbs the compositional residual. When no finite-dimensional basis suffices at all (Arnold cat map), no readout choice recovers sparse structure, and the flat eigenvalue profile becomes the diagnostic.

\paragraph{Position in the pipeline.}
Deep-Koopman-KANDy is not a replacement for SINDy, KANDy, or EDMD as a
downstream regression method. It is a way to learn the candidate observable
structure that those methods need as input. The architecture and
$h'(g) = (\nabla f\cdot\nabla g)/\|\nabla g\|^2$ recovery procedure form a
principled route to KAN symbolic function recovery, and the pruning step
sharpens dictionary parsimony.

\paragraph{Limitations.}
The method still requires a basis for the symbolic readout, even though the choice can be deferred until after training. Removing all candidate-library dependence---learning the readout basis itself---is left to future work. The Arnold cat map result also bounds what any finite-dimensional Koopman model can do: when the spectrum is purely continuous, no architecture rescues long-horizon prediction beyond the Lyapunov time.

\bibliographystyle{unsrtnat}
\bibliography{references}

@article{koopman1931,
  author  = {Koopman, B. O.},
  title   = {Hamiltonian Systems and Transformation in Hilbert Space},
  journal = {Proceedings of the National Academy of Sciences of the United States of America},
  year    = {1931},
  volume  = {17},
  number  = {5},
  pages   = {315--318},
  doi     = {10.1073/pnas.17.5.315}
}

@article{koopman1932,
  author  = {Koopman, B. O. and von Neumann, J.},
  title   = {Dynamical Systems of Continuous Spectra},
  journal = {Proceedings of the National Academy of Sciences of the United States of America},
  volume  = {18},
  number  = {3},
  pages   = {255--263},
  year    = {1932},
  doi     = {10.1073/pnas.18.3.255}
}

@article{lusch2018,
  author  = {Lusch, Bethany and Kutz, J. Nathan and Brunton, Steven L.},
  title   = {Deep Learning for Universal Linear Embeddings of Nonlinear Dynamics},
  journal = {Nature Communications},
  year    = {2018},
  volume  = {9},
  number  = {1},
  pages   = {4950},
  doi     = {10.1038/s41467-018-07210-0}
}

@article{li2017,
  author  = {Li, Qianxiao and Dietrich, Felix and Bollt, Erik M. and Kevrekidis, Ioannis G.},
  title   = {Extended Dynamic Mode Decomposition with Dictionary Learning: A Data-Driven Adaptive Spectral Decomposition of the Koopman Operator},
  journal = {Chaos},
  year    = {2017},
  volume  = {27},
  number  = {10},
  pages   = {103111},
  doi     = {10.1063/1.4993854}
}

@article{otto2019,
  author  = {Otto, Samuel E. and Rowley, Clarence W.},
  title   = {Linearly-Recurrent Autoencoder Networks for Learning Dynamics},
  journal = {SIAM Journal on Applied Dynamical Systems},
  year    = {2019},
  volume  = {18},
  number  = {1},
  pages   = {558--593},
  doi     = {10.1137/18M1177846}
}

@article{kandy,
  author        = {Slote, Kevin and Fish, Jeremie and Bollt, Erik},
  title         = {{KANDy}: Kolmogorov--Arnold Networks and Dynamical System Discovery},
  journal       = {arXiv preprint arXiv:2602.20413},
  year          = {2026},
  doi           = {10.48550/arXiv.2602.20413},
  archivePrefix = {arXiv},
  eprint        = {2602.20413},
  primaryClass  = {math.DS}
}

@article{MoradiBollt2026,
  author        = {Moradi, Mohammadamin and Panahi, Shirin and Bollt, Erik and Lai, Ying-Cheng},
  title         = {Kolmogorov-Arnold Network Autoencoders},
  journal       = {arXiv preprint arXiv:2410.02077},
  year          = {2024},
  doi           = {10.48550/arXiv.2410.02077},
  archivePrefix = {arXiv},
  eprint        = {2410.02077},
  primaryClass  = {cs.LG}
}

@article{schmid2010dynamic,
  author    = {Peter J. Schmid},
  title     = {Dynamic mode decomposition of numerical and experimental data},
  journal   = {Journal of Fluid Mechanics},
  volume    = {656},
  pages     = {5--28},
  year      = {2010},
  doi       = {10.1017/S0022112010001217},
  publisher = {Cambridge University Press}
}

@article{Williams2014,
  author = {Williams, Matthew and Kevrekidis, Ioannis and Rowley, Clarence},
  year = {2014},
  month = {08},
  pages = {},
  title = {A Data-Driven Approximation of the Koopman Operator: Extending Dynamic Mode Decomposition},
  volume = {25},
  journal = {Journal of Nonlinear Science},
  doi = {10.1007/s00332-015-9258-5}
}

@article{brunton2022modern,
  author    = {Steven L. Brunton and Marko Budi{\v{s}}i{\'c} and Eurika Kaiser and J. Nathan Kutz},
  title     = {Modern Koopman Theory for Dynamical Systems},
  journal   = {SIAM Review},
  volume    = {64},
  number    = {2},
  pages     = {229--340},
  year      = {2022},
  month     = may,
  doi       = {10.1137/21M1401243},
  publisher = {Society for Industrial and Applied Mathematics}
}

@article{mezic2013fluids,
  author  = {Mezi{\'c}, Igor},
  title   = {Analysis of Fluid Flows via Spectral Properties of the Koopman Operator},
  journal = {Annual Review of Fluid Mechanics},
  volume  = {45},
  pages   = {357--378},
  year    = {2013},
  doi     = {10.1146/annurev-fluid-011212-140652}
}

@article{rowley2009spectral,
  author  = {Rowley, Clarence W. and Mezi{\'c}, Igor and Bagheri, Shervin and Schlatter, Philipp and Henningson, Dan S.},
  title   = {Spectral analysis of nonlinear flows},
  journal = {Journal of Fluid Mechanics},
  volume  = {641},
  pages   = {115--127},
  year    = {2009},
  doi     = {10.1017/S0022112009992059}
}

@article{williams2016kernel,
  author  = {Williams, Matthew O. and Rowley, Clarence W. and Kevrekidis, Ioannis G.},
  title   = {A kernel-based method for data-driven Koopman spectral analysis},
  journal = {Journal of Computational Dynamics},
  volume  = {2},
  number  = {2},
  pages   = {247--265},
  year    = {2015},
  doi     = {10.3934/jcd.2015005}
}

@article{budisic2012applied,
  author  = {Budi{\v{s}}i{\'c}, Marko and Mohr, Ryan and Mezi{\'c}, Igor},
  title   = {Applied Koopmanism},
  journal = {Chaos},
  year    = {2012},
  volume  = {22},
  pages   = {047510},
  doi     = {10.1063/1.4772195},
}

@article{williams2015datafusion,
  author  = {Williams, Matthew O. and Rowley, Clarence W. and Mezi{\'c}, Igor and Kevrekidis, Ioannis G.},
  title   = {Data fusion via intrinsic dynamic variables: An application of data-driven Koopman spectral analysis},
  journal = {Europhysics Letters},
  year    = {2015},
  volume  = {109},
  pages   = {40007},
  doi     = {10.1209/0295-5075/109/40007},
}

@article{brunton2016koopman,
  author  = {Brunton, Steven L. and Brunton, Bingni W. and Proctor, Joshua L. and Kutz, J. Nathan},
  title   = {Koopman invariant subspaces and finite linear representation of nonlinear dynamical systems for control},
  journal = {PLOS ONE},
  year    = {2016},
  volume  = {11},
  pages   = {e0150171},
  doi     = {10.1371/journal.pone.0150171},
}

@article{korda2017linearpredictors,
  author  = {Korda, Milan and Mezi{\'c}, Igor},
  title   = {Linear predictors for nonlinear dynamical systems: Koopman operator meets model predictive control},
  journal = {arXiv preprint arXiv:1703.10112},
  year    = {2017},
  eprint  = {1703.10112},
  archivePrefix = {arXiv},
  primaryClass  = {math.OC},
}

@article{sharma2016correspondence,
  author  = {Sharma, A. S. and Mezi{\'c}, Igor and McKeon, B. J.},
  title   = {Correspondence between Koopman mode decomposition, resolvent mode decomposition, and invariant solutions of the Navier--Stokes equations},
  journal = {Physical Review Fluids},
  year    = {2016},
  volume  = {1},
  pages   = {032402},
  doi     = {10.1103/PhysRevFluids.1.032402},
}

@article{georgescu2015building,
  author  = {Georgescu, M. and Mezi{\'c}, Igor},
  title   = {Building energy modeling: A systematic approach to zoning and model reduction using Koopman mode analysis},
  journal = {Energy and Buildings},
  year    = {2015},
  volume  = {86},
  pages   = {794--802},
  doi     = {10.1016/j.enbuild.2014.10.046},
}

@inproceedings{giannakis2015spatiotemporal,
  author    = {Giannakis, Dimitrios and Slawinska, Joanna and Zhao, Zhizhen},
  title     = {Spatiotemporal feature extraction with data-driven Koopman operators},
  booktitle = {Proceedings of the 32nd International Conference on Machine Learning},
  series    = {JMLR Workshop and Conference Proceedings},
  volume    = {44},
  pages     = {103--115},
  year      = {2015},
}

@article{wu2017variational,
  author  = {Wu, H. and N{\"u}ske, F. and Paul, F. and Klus, S. and Koltai, P. and No{\'e}, F.},
  title   = {Variational Koopman models: Slow collective variables and molecular kinetics from short off-equilibrium simulations},
  journal = {The Journal of Chemical Physics},
  year    = {2017},
  volume  = {146},
  pages   = {154104},
  doi     = {10.1063/1.4979344},
}

@article{oxfordgenomicskan2025,
	title        = {{K}olmogorov--{A}rnold networks for genomic tasks},
	author       = {Oleksandr Cherednichenko and Maria Poptsova},
	year         = 2025,
	journal      = {Bioinformatics},
	volume       = 41,
	number       = 2,
	pages        = {412--424}
}

@article{naturegraphkan2025,
	title    = {{K}olmogorov--{A}rnold graph neural networks for molecular property prediction},
    author  = {Li, Longlong and Zhang, Yipeng and Wang, Guanghui and Xia, Kelin},
    journal = {Nature Machine Intelligence},
    volume  = {7},
    number  = {8},
    pages   = {1346--1354},
    year    = {2025},
    doi     = {10.1038/s42256-025-01087-7},
    url     = {https://doi.org/10.1038/s42256-025-01087-7}
}

@article{scidirectkan2025,
	title        = {Opening the {AI} black-box: Symbolic regression with {K}olmogorov--{A}rnold Networks},
	author       = {Nataly R. Panczyk, Omer F. Erdem, Majdi I. Radaideh},
	year         = 2025,
	journal      = {Energy AI},
	volume       = 22,
	pages        = 100258
}

@article{lorenz1963deterministic,
	title        = {Deterministic Nonperiodic Flow},
	author       = {Lorenz, Edward N.},
	year         = 1963,
	journal      = {Journal of the Atmospheric Sciences},
	volume       = 20,
	number       = 2,
	pages        = {130--141},
	doi          = {10.1175/1520-0469(1963)020<0130:DNF>2.0.CO;2}
}

@article{coxkan2025,
    author = {Knottenbelt, William and McGough, William and Wray, Rebecca and Zhang, Woody Zhidong and Liu, Jiashuai and Machado, Ines Prata and Gao, Zeyu and Crispin-Ortuzar, Mireia},
    title = {Cox{KAN}: {K}olmogorov--{A}rnold networks for interpretable, high-performance survival analysis},
    journal = {Bioinformatics},
    volume = {41},
    number = {8},
    pages = {btaf413},
    year = {2025},
    month = {07},
    issn = {1367-4811},
    doi = {10.1093/bioinformatics/btaf413},
}

@article{Bakarji2023DelayAutoencoders,
  author  = {Bakarji, Joseph and Champion, Kathleen and Kutz, J. Nathan and Brunton, Steven L.},
  title   = {Discovering governing equations from partial measurements with deep delay autoencoders},
  journal = {Proceedings of the Royal Society A},
  year    = {2023},
  volume  = {479},
  number  = {2276},
  pages   = {20230422},
  month   = aug,
  doi     = {10.1098/rspa.2023.0422},
}

@article{FS1987,
  author  = {Farmer, J. D. and Sidorowich, J. J.},
  title   = {Predicting chaotic time series},
  journal = {Physical Review Letters},
  volume  = {59},
  pages   = {845},
  year    = {1987},
  doi     = {10.1103/PhysRevLett.59.845}
}

@article{CM1987,
  author  = {Crutchfield, J. P. and McNamara, B.},
  title   = {Equations of motion from a data series},
  journal = {Complex Systems},
  volume  = {1},
  pages   = {417},
  year    = {1987}
}

@article{Casdagli1989,
  author  = {Casdagli, M.},
  title   = {Nonlinear prediction of chaotic time series},
  journal = {Physica D},
  volume  = {35},
  pages   = {335},
  year    = {1989}
}

@article{SGMCPW1990,
  author  = {Sugihara, G. and Grenfell, B. and May, R. M. and Chesson, P. and Platt, H. M. and Williamson, M.},
  title   = {Distinguishing error from chaos in ecological time series},
  journal = {Philosophical Transactions of the Royal Society of London. Series B},
  volume  = {330},
  pages   = {235},
  year    = {1990}
}

@article{GS1990,
  author  = {Grassberger, P. and Schreiber, T.},
  title   = {Nonlinear time sequence analysis},
  journal = {International Journal of Bifurcation and Chaos},
  volume  = {1},
  pages   = {521},
  year    = {1990}
}

@article{Gouesbet1991,
  author  = {Gouesbet, G.},
  title   = {Reconstruction of standard and inverse vector fields equivalent to a R\"ossler system},
  journal = {Physical Review A},
  volume  = {44},
  pages   = {6264},
  year    = {1991},
  doi     = {10.1103/PhysRevA.44.6264}
}

@article{BBBB1992,
  author  = {Baake, E. and Baake, M. and Bock, H.-G. and Briggs, K. M.},
  title   = {Fitting ordinary differential equations to chaotic data},
  journal = {Physical Review A},
  volume  = {45},
  pages   = {5524},
  year    = {1992},
  doi     = {10.1103/PhysRevA.45.5524}
}

@article{Sauer1994_PRL,
  author  = {Sauer, T.},
  title   = {Reconstruction of dynamical systems from interspike intervals},
  journal = {Physical Review Letters},
  volume  = {72},
  pages   = {3811},
  year    = {1994},
  doi     = {10.1103/PhysRevLett.72.3811}
}

@article{Parlitz1996,
  author  = {Parlitz, U.},
  title   = {Estimating model parameters from time series by autosynchronization},
  journal = {Physical Review Letters},
  volume  = {76},
  pages   = {1232},
  year    = {1996}
}

@article{Szpiro1997,
  author  = {Szpiro, G. G.},
  title   = {Forecasting chaotic time series with genetic algorithms},
  journal = {Physical Review E},
  volume  = {55},
  pages   = {2557},
  year    = {1997},
  doi     = {10.1103/PhysRevE.55.2557}
}

@article{HKS1999,
  author  = {Hegger, R. and Kantz, H. and Schreiber, T.},
  title   = {Practical implementation of nonlinear time series methods: The tisean package},
  journal = {Chaos},
  volume  = {9},
  pages   = {413},
  year    = {1999}
}

@article{Bollt2000,
  author  = {Bollt, E. M.},
  title   = {Controlling chaos and the inverse {Frobenius}-{Perron} problem: Global stabilization of arbitrary invariant measures},
  journal = {International Journal of Bifurcation and Chaos},
  volume  = {10},
  pages   = {1033},
  year    = {2000}
}

@article{YB2007,
  author  = {Yao, C. and Bollt, E. M.},
  title   = {Modeling and nonlinear parameter estimation with {Kronecker} product representation for coupled oscillators and spatiotemporal systems},
  journal = {Physica D},
  volume  = {227},
  pages   = {78},
  year    = {2007}
}

@article{TZJ2007,
  author  = {Tao, C. and Zhang, Y. and Jiang, J. J.},
  title   = {Estimating system parameters from chaotic time series with synchronization optimized by a genetic algorithm},
  journal = {Physical Review E},
  volume  = {76},
  pages   = {016209},
  year    = {2007},
  doi     = {10.1103/PhysRevE.76.016209}
}

@article{WYLKG2011a,
  author = {Wang, Wen-Xu and Yang, Rui and Lai, Ying-Cheng and Kovanis, Vassilios and Grebogi, Celso},
  title   = {Predicting catastrophes in nonlinear dynamical systems by compressive sensing},
  journal = {Physical Review Letters},
  volume = {106},
  issue = {15},
  pages = {154101},
  numpages = {4},
  year = {2011},
  month = {Apr},
  publisher = {American Physical Society},
  doi = {10.1103/PhysRevLett.106.154101},
}

@article{WLGY2011b,
 title = {Network Reconstruction Based on Evolutionary-Game Data via Compressive Sensing},
  author = {Wang, Wen-Xu and Lai, Ying-Cheng and Grebogi, Celso and Ye, Jieping},
  journal = {Physical Review X},
  volume = {1},
  issue = {2},
  pages = {021021},
  numpages = {7},
  year = {2011},
  month = {Dec},
  publisher = {American Physical Society},
  doi = {10.1103/PhysRevX.1.021021},
}

@article{SWL2012b,
  title = {Detecting hidden nodes in complex networks from time series},
  author = {Su, Ri-Qi and Wang, Wen-Xu and Lai, Ying-Cheng},
  journal = {Phys. Rev. E},
  volume = {85},
  issue = {6},
  pages = {065201},
  numpages = {4},
  year = {2012},
  month = {Jun},
  publisher = {American Physical Society},
  doi = {10.1103/PhysRevE.85.065201},
}

@article{SWFDL2014,
  author  = {Shen, Z. and Wang, W.-X. and Fan, Y. and Di, Z. and Lai, Y.-C.},
  title   = {Reconstructing propagation networks with natural diversity and identifying hidden sources},
  journal = {Nature Communications},
  volume  = {5},
  pages   = {4323},
  year    = {2014}
}

@article{YLG2012,
  author  = {Rui Yang and Ying-Cheng Lai and Celso Grebogi},
  title   = {Forecasting the future: Is it possible for adiabatically time-varying nonlinear dynamical systems?},
  journal = {Chaos},
  volume  = {22},
  number  = {3},
  pages   = {033119},
  year    = {2012},
  doi     = {10.1063/1.4740057}
}

@article{Rossler1976,
  author  = {R\"ossler, O. E.},
  title   = {Equation for continuous chaos},
  journal = {Physics Letters A},
  volume  = {57},
  pages   = {397},
  year    = {1976}
}

@article{HJM1985,
  author  = {Hammel, S. M. and Jones, C. K. R. T. and Moloney, J. V.},
  title   = {Global dynamical behavior of the optical field in a ring cavity},
  journal = {Journal of the Optical Society of America B},
  volume  = {2},
  pages   = {552},
  year    = {1985}
}

@article{Holling1959a,
  author  = {Holling, C. S.},
  title   = {The components of predation as revealed by a study of small-mammal predation of the european pine sawfly},
  journal = {The Canadian Entomologist},
  volume  = {91},
  pages   = {293--320},
  year    = {1959}
}

@article{Holling1959b,
  author  = {Holling, C. S.},
  title   = {Some characteristics of simple types of predation and parasitism},
  journal = {The Canadian Entomologist},
  volume  = {91},
  pages   = {385},
  year    = {1959}
}

@article{JHSLGHL2018,
  author  = {Jiang, J. and Huang, Z.-G. and Seager, T. P. and Lin, W. and Grebogi, C. and Hastings, A. and Lai, Y.-C.},
  title   = {Predicting tipping points in mutualistic networks through dimension reduction},
  journal = {Proceedings of the National Academy of Sciences of the USA},
  volume  = {115},
  pages   = {E639},
  year    = {2018}
}

@article{Kolmogorov1957,
  author  = {Kolmogorov, A. N.},
  title   = {On the representation of continuous functions of many variables by superposition of continuous functions of one variable and addition},
  journal = {Doklady Akademii Nauk SSSR},
  volume  = {114},
  pages   = {953},
  year    = {1957}
}

@incollection{Arnold2009_Kolmogorov,
  author    = {Kolmogorov, A. N.},
  title     = {On the representation of functions of several variables as a superposition of functions of a smaller number of variables},
  booktitle = {Collected Works: Representations of Functions, Celestial Mechanics and KAM Theory, 1957--1965},
  editor    = {Givental, A. B. and Khesin, B. A. and Marsden, J. E. and Varchenko, A. N. and Vassiliev, V. A. and Viro, O. Y. and Zakalyukin, V. M.},
  publisher = {Springer Berlin Heidelberg},
  address   = {Berlin, Heidelberg},
  pages     = {25--46},
  year      = {2009},
  doi       = {10.1007/978-3-642-01742-1_5}
}

@article{Candes2006Robust,
  author    = {Emmanuel J. Cand{\`e}s and Justin Romberg and Terence Tao},
  title     = {Robust uncertainty principles: {E}xact signal reconstruction from highly incomplete frequency information},
  journal   = {IEEE Transactions on Information Theory},
  volume    = {52},
  number    = {2},
  pages     = {489--509},
  year      = {2006},
  doi       = {10.1109/TIT.2005.862083}
}

@article{Candes2006Stable,
  author    = {Emmanuel J. Cand{\`e}s and Justin Romberg and Terence Tao},
  title     = {Stable signal recovery from incomplete and inaccurate measurements},
  journal   = {Communications on Pure and Applied Mathematics},
  volume    = {59},
  number    = {8},
  pages     = {1207--1223},
  year      = {2006},
  doi       = {10.1002/cpa.20124}
}

@article{Donoho2006Compressed,
  author    = {David L. Donoho},
  title     = {Compressed sensing},
  journal   = {IEEE Transactions on Information Theory},
  volume    = {52},
  number    = {4},
  pages     = {1289--1306},
  year      = {2006},
  doi       = {10.1109/TIT.2006.871582}
}

@article{Baraniuk2007Compressed,
  author    = {Richard G. Baraniuk},
  title     = {Compressed sensing},
  journal   = {IEEE Signal Processing Magazine},
  volume    = {24},
  number    = {4},
  pages     = {118--121},
  year      = {2007},
  doi       = {10.1109/MSP.2007.4286571}
}

@article{Candes2008Introduction,
  author    = {Emmanuel J. Cand{\`e}s and Michael B. Wakin},
  title     = {An introduction to compressive sampling},
  journal   = {IEEE Signal Processing Magazine},
  volume    = {25},
  number    = {2},
  pages     = {21--30},
  year      = {2008},
  doi       = {10.1109/MSP.2007.914731}
}

@article{Bagrow2025,
doi = {10.1088/2632-2153/adf9bd},
year = {2025},
month = {aug},
publisher = {IOP Publishing},
volume = {6},
number = {3},
pages = {035037},
author = {Bagrow, James and Bongard, Josh},
title = {Multi-exit {K}olmogorov–{A}rnold networks: enhancing accuracy and parsimony},
journal = {Machine Learning: Science and Technology}
}

@article{Koenig2025,
title = {{LeanKAN}: a parameter-lean {K}olmogorov-{A}rnold network layer with improved memory efficiency and convergence behavior},
journal = {Neural Networks},
volume = {192},
pages = {107883},
year = {2025},
issn = {0893-6080},
doi = {https://doi.org/10.1016/j.neunet.2025.107883},
author = {Benjamin C. Koenig and Suyong Kim and Sili Deng}
}

@article{Koenig2024,
title = {{KAN}-{ODEs}: {K}olmogorov–{A}rnold network ordinary differential equations for learning dynamical systems and hidden physics},
journal = {Computer Methods in Applied Mechanics and Engineering},
volume = {432},
pages = {117397},
year = {2024},
issn = {0045-7825},
doi = {https://doi.org/10.1016/j.cma.2024.117397},
author = {Benjamin C. Koenig and Suyong Kim and Sili Deng}
}

@inproceedings{liu2025kan,
    title={{KAN}: Kolmogorov Arnold Networks},
    author={Ziming Liu and Yixuan Wang and Sachin Vaidya and Fabian Ruehle and James Halverson and Marin Soljacic and Thomas Y. Hou and Max Tegmark},
    booktitle={The Thirteenth International Conference on Learning Representations},
    year={2025},
    url={https://openreview.net/forum?id=Ozo7qJ5vZi}
}

@incollection{Kuramoto1975SelfEntrapment,
  author    = {Kuramoto, Yoshiki},
  title     = {Self-entrainment of a population of coupled non-linear oscillators},
  booktitle = {International Symposium on Mathematical Problems in Theoretical Physics},
  editor    = {Araki, Huzihiro},
  series    = {Lecture Notes in Physics},
  volume    = {39},
  publisher = {Springer},
  address   = {Berlin, Heidelberg},
  year      = {1975},
  doi       = {10.1007/BFb0013365}
}

@article{Basiri2025SINDyG,
  author  = {Basiri, Mohammad Amin and Khanmohammadi, Sina},
  title   = {{SINDyG}: sparse identification of nonlinear dynamical systems from graph-structured data, with applications to {S}tuart--{L}andau oscillator networks},
  journal = {Journal of Complex Networks},
  volume  = {13},
  number  = {5},
  year    = {2025},
  month   = sep,
  pages   = {cnaf029},
  doi     = {10.1093/comnet/cnaf029},
}

@article{sindy,
author = {Steven L. Brunton  and Joshua L. Proctor  and J. Nathan Kutz },
title = {Discovering governing equations from data by sparse identification of nonlinear dynamical systems},
journal = {Proceedings of the National Academy of Sciences},
volume = {113},
number = {15},
pages = {3932-3937},
year = {2016},
doi = {10.1073/pnas.1517384113},
eprint = {https://www.pnas.org/doi/pdf/10.1073/pnas.1517384113},
}

@article{Panahi2025KANModelDiscovery,
  author    = {Panahi, Shirin and Moradi, Mohammadamin and Bollt, Erik M. and Lai, Ying-Cheng},
  title     = {Data-driven model discovery with {K}olmogorov-{A}rnold networks},
  journal   = {Physical Review Research},
  volume    = {7},
  pages     = {023037},
  year      = {2025},
  month     = {April},
  doi       = {10.1103/PhysRevResearch.7.023037},
  publisher = {American Physical Society}
}

@article{sindy_ae,
author = {Kathleen Champion  and Bethany Lusch  and J. Nathan Kutz  and Steven L. Brunton },
title = {Data-driven discovery of coordinates and governing equations},
journal = {Proceedings of the National Academy of Sciences},
volume = {116},
number = {45},
pages = {22445-22451},
year = {2019},
doi = {10.1073/pnas.1906995116}
}

@article{ikeda1979,
  author  = {Ikeda, Kensuke},
  title   = {Multiple-valued stationary state and its instability of the transmitted light by a ring cavity system},
  journal = {Optics Communications},
  volume  = {30},
  number  = {2},
  pages   = {257--261},
  year    = {1979},
  doi     = {10.1016/0030-4018(79)90090-7}
}

\appendix
\section{Related Work}

\paragraph{Koopman Operator Theory and Data-Driven Approximations.} The Koopman operator~\cite{koopman1931, koopman1932} provides a linear, infinite-dimensional description of nonlinear dynamical systems and has been applied to model reduction, prediction, data fusion, and control across fluid dynamics~\cite{mezic2013fluids, sharma2016correspondence, rowley2009spectral}, energy systems~\cite{georgescu2015building}, oceanography~\cite{giannakis2015spatiotemporal}, and molecular kinetics~\cite{wu2017variational}. Dynamic Mode Decomposition (DMD)~\cite{schmid2010dynamic} offers a practical finite-dimensional Koopman approximation, with theoretical foundations surveyed in~\cite{brunton2022modern, budisic2012applied}. Extended DMD (EDMD)~\cite{Williams2014} improves accuracy by lifting the state through a hand-crafted dictionary of nonlinear observables; replacing this fixed dictionary with a trainable neural-network dictionary achieves strong reconstruction without manual feature engineering~\cite{MoradiBollt2026}. Linear Koopman predictors are studied in~\cite{korda2017linearpredictors}, and data fusion applications are explored in~\cite{williams2015datafusion, brunton2016koopman}. Deep Koopman methods~\cite{lusch2018, li2017, otto2019} learn observable functions via neural-network autoencoders, achieving flexible lifting at the cost of interpretability; delay-coordinate autoencoders~\cite{Bakarji2023DelayAutoencoders} partially address identifiability in this setting.

\paragraph{Data-Driven System Identification.} Identifying nonlinear dynamical systems from data has a long history\cite{FS1987, CM1987, Casdagli1989, SGMCPW1990, GS1990, HKS1999, Bollt2000, YB2007, WYLKG2011a, WLGY2011b, SWL2012b, SWFDL2014}, encompassing linear approximations~\cite{FS1987, Gouesbet1991, Sauer1994_PRL}, fitting of differential equations~\cite{BBBB1992}, chaotic synchronization~\cite{Parlitz1996}, genetic algorithms~\cite{Szpiro1997, TZJ2007}, the inverse Frobenius--Perron method~\cite{Bollt2000}, and least-squares approaches~\cite{YB2007}. Sparse optimization methods assume that governing equations have a sparse representation in a chosen basis~\cite{WYLKG2011a, YLG2012}, connecting naturally to compressive sensing~\cite{Candes2008Introduction, Baraniuk2007Compressed, Donoho2006Compressed, Candes2006Stable, Candes2006Robust}. SINDy~\cite{sindy} operationalizes sparse regression for continuous systems and performs well on the Lorenz~\cite{lorenz1963deterministic} and R\"{o}ssler~\cite{Rossler1976} systems. Sparsity-based methods nevertheless fail when governing equations lack a sparse polynomial structure, as in the Ikeda map~\cite{Ikeda1979, HJM1985}, Holling-type ecological models~\cite{Holling1959a, Holling1959b, JHSLGHL2018}, and the Kuramoto model~\cite{Kuramoto1975SelfEntrapment, Basiri2025SINDyG}.

\paragraph{Kolmogorov--Arnold Networks.} KANs~\cite{liu2025kan} replace fixed nonlinear activations with learnable univariate spline functions on each edge, grounded in the Kolmogorov--Arnold representation theorem~\cite{Kolmogorov1957, Arnold2009_Kolmogorov}. This design directly exposes the functional form of the learned mapping. KANs have demonstrated competitive or superior performance relative to MLPs on regression, PDE solving, genomics, and scientific machine learning tasks, often with improved interpretability and symbolic extractability~\cite{coxkan2025, oxfordgenomicskan2025, naturegraphkan2025, scidirectkan2025, liu2025kan}. 

\paragraph{KANs for Dynamical Systems.} Several recent works apply KANs to ODE and dynamical system identification. Panahi et al.~\cite{Panahi2025KANModelDiscovery} demonstrate that KANs can reconstruct attractor statistics for discrete dynamical systems, though without a principled method for recovering governing equations. Bagrow~\cite{Bagrow2025} introduces multi-exit KANs with a novel loss function for improved prediction accuracy in dynamical settings. Koenig~\cite{Koenig2024} applies symbolic regression to KAN activation weights to recover the Lotka--Volterra predator--prey model (LEAN-KAN); subsequent work~\cite{Koenig2025} refines equation estimation via an improved multiplication layer. KANDy~\cite{kandy} replaces sparse regression in SINDy with a single KAN layer, enabling symbolic extraction of univariate observables.

\section{Lorenz}
\label{app:lorenz_eigs}


The Lorenz system in Equation~\ref{eq:lorenz} is a canonical three-dimensional nonlinear system exhibiting sensitive dependence on initial conditions and a strange attractor. The Lorenz system is a common benchmark in data-driven dynamical systems \cite{sindy, kandy, sindy_ae}. It therefore provides a stringent benchmark for evaluating whether the Deep-Koopman-KANDy representation can recover the sparse nonlinear structure of the governing equations.

Table~\ref{tab:lorenz_results} summarizes the performance of the best learned model. The stable Koopman parameterization yields accurate one-step prediction and reconstruction while maintaining all learned eigenvalues inside the unit circle. Despite the intrinsic chaotic divergence of nearby trajectories, the model achieves prediction horizons on the order of one Lyapunov time, indicating that the learned latent dynamics capture the dominant short-time evolution before chaotic error growth dominates.

Beyond forecasting accuracy, the learned representation also recovers the algebraic structure of the Lorenz vector field. Using level-set decomposition (Proposition~\ref{prop:hprime}) of the encoder coordinates followed by sparse polynomial regression, every target term in the Lorenz dictionary, $\{x,y,z,xy,xz\}$, is identified in at least one latent coordinate. Table~\ref{tab:levelset-decomposition} gives the corresponding per-coordinate decomposition, showing that the dominant recovered terms are linear, while the nonlinear interaction terms appear with smaller but consistent coefficients. 
\begin{table}[ht]
      \centering
      \begin{tabular}{@{}lc@{}}
      \toprule
      Metric & Value \\
      \midrule
      Architecture & PyKAN $[3, 5, 5]$ + $K = \Omega -L^\top L$ \\
      Parameters & 1{,}593 \\
      Pruning & Edge threshold $0.03$, retrain 100 epochs \\
      \midrule
      One-step MSE & $2.55 \times 10^{-3}$ \\
      Val loss (3-term) & $0.013$ \\
      Reconstruction MSE & $5.42 \times 10^{-3}$ \\
      Horizon (NRMSE $> 0.5$) & $0.97 \pm 0.47\;\tau_\Lambda$ \\
      Horizon (NRMSE $> 1.0$) & $1.37 \pm 0.81\;\tau_\Lambda$ \\
      \midrule
      $|\lambda|$ range & $[0.854,\; 0.997]$ \\
      $\max\;\text{Re}(\mu)$ & $-0.26$ \\
      Eigenmodes & $\pm 1.43$ Hz (spiral), 3 real damping \\
      \midrule
      Dictionary (Lasso $\lambda = 1\!\times\!10^{-6}$, degree 3) & $\{x,\, y,\, z,\, xy,\, xz\}$ \\
      Union recall & $1.0$ \\
      Per-seed Jaccard (mean $\pm$ std) & $0.79 \pm 0.06$ \\
      Mean FP / dim & $2.2$ \\
      \bottomrule
      \end{tabular}
      \caption{Performance summary of the best Lorenz model: PyKAN $[3,5,5]$ with
      stable Koopman propagator, after edge pruning (threshold $0.03$) and retraining.
      Every target term in $\{x,\,y,\,z,\,xy,\,xz\}$ is recovered by at least one latent
      coordinate (union recall $= 1.0$); per-seed Jaccard is $0.79 \pm 0.06$ across 5 seeds,
      with residual false positives at coefficients $\lesssim 10^{-2}$ consistent with the
      codimension residual predicted in Appendix~\ref{app:attractor}. Recovery uses
      level-set decomposition with degree-3 polynomial inner function and Lasso
      $\lambda = 1\!\times\!10^{-6}$.
      Prediction horizon averaged over 12 validation trajectories;
      $\tau_\Lambda \approx 1.1\;\text{s}$.}
      \label{tab:lorenz_results}
  \end{table}

\begin{table}[ht]
    \centering
    \begin{tabular}{@{}clrrrrrc@{}}
    \toprule
    Dim & $R^2(h \circ g)$ & $x$ & $y$ & $z$ & $xy$ & $xz$ & FP \\
    \midrule
    $z_0$ & $0.924$ & $+0.524$ & $-0.121$ & $-0.018$ & $-0.008$ & $-0.025$ & --- \\
    $z_1$ & $0.584$ & $-0.393$ & $+0.322$ & --- & $+0.033$ & --- & $y^2,\, x^2$ \\
    $z_2$ & $0.962$ & $+0.446$ & $-0.082$ & $+0.187$ & --- & $-0.028$ & $x^2$ \\
    $z_3$ & $0.947$ & $+0.208$ & $+0.065$ & $+0.007$ & --- & $-0.009$ & $x^2$ \\
    $z_4$ & $0.826$ & $-0.234$ & $+0.032$ & $+0.246$ & $+0.028$ & --- & $x^2$ \\
    \bottomrule
    \end{tabular}
    \caption{Level-set decomposition $z_k \approx h \circ g$ for the
    Deep-Koopman-KANDy encoder (Lasso $\lambda = 10^{-5}$, degree-3
    polynomial basis). Columns show the top-5 monomial coefficients of the
    inner function~$g$; dashes indicate terms outside the top~5 or below
    $10^{-3}$. Every target term $\{x,\,y,\,z,\,xy,\,xz\}$ appears in at least one
    dimension. Cross-term coefficients are $1$--$2$ orders of magnitude below
    the linear terms but consistently recovered. FP: false-positive monomials
    in the top~5.}
    \label{tab:levelset-decomposition}
\end{table}

\clearpage
\section{Ikeda}\label{sec:ikeda}
\begin{figure}
    \centering
    \includegraphics[width=1.0\linewidth]{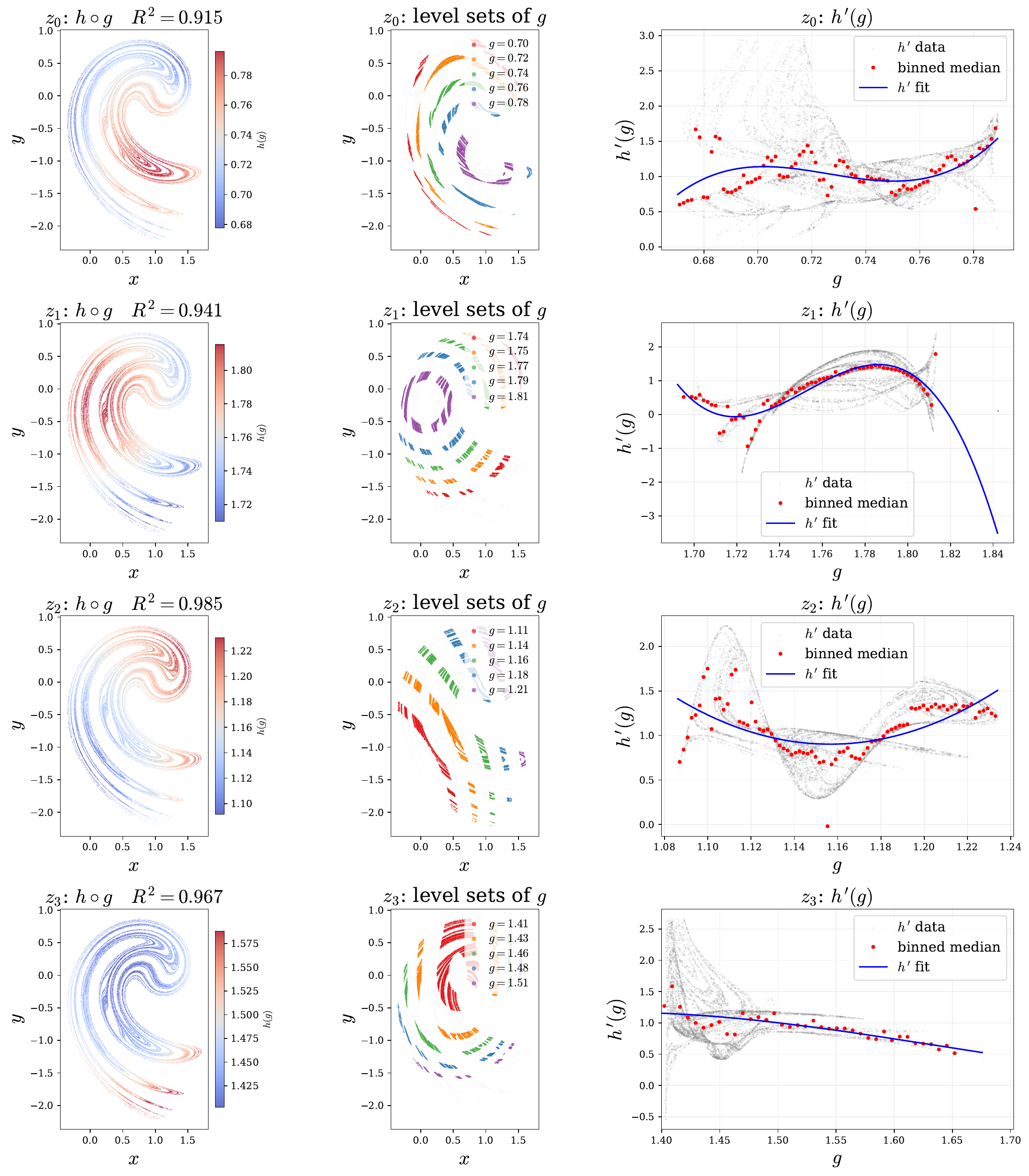}
    \caption{\textbf{Ikeda map: level-set decomposition with a misspecified polynomial readout dictionary.} The Ikeda map admits no sparse polynomial representation, so the degree-3 polynomial Lasso used to recover the inner function $g$ is misspecified by construction. We use it anyway, treating the polynomial as a flexible interpolant rather than a structural prior. \textbf{(left)} Attractor colored by the reconstruction $h(g(\mathbf{x}))$ for each of four latent dimensions $z_0,\dots,z_3$ of the trained $[2,4,4]$ encoder; $R^2 \in [0.92, 0.99]$ across all four dimensions. \textbf{(center)} Level sets of the recovered $g$ trace coherent bands across the attractor, indicating that the polynomial captures the encoder's foliation structure despite its misspecification for the underlying dynamics. \textbf{(right)} Outer derivatives $h'(g)$ recovered via the gradient identity (Eq.~\ref{eq:hprime}). All four dimensions yield nontrivial $h'$ structure --- the gradient identity recovers nonlinear outer functions from the polynomial residual rather than collapsing to affine fallback. The $z_2$ panel is the clearest case: the binned medians (red) show two-peak structure that the cubic fit (blue) tracks tightly. Together, the four dimensions demonstrate the central decoupling claim: a misspecified readout dictionary degrades parsimony, not expressivity, because compositional structure recovered through $h$ compensates for what $g$ alone cannot represent.}
    \label{fig:iked_minimal}
\end{figure}

\begin{table}[ht]
      \centering
      \begin{tabular}{@{}clrrrrrrc@{}}
      \toprule
      Dim & $R^2(h \circ g)$ & $x$ & $y$ & $x^2$ & $y^2$ & $xy$ & $h$ affine? & FP \\
      \midrule
      $z_0$ & $0.915$ & $+0.110$ & $-0.033$ & $-0.081$ & $-0.011$ & $-0.038$ & no & $y^3,\, x^3,\, x^2y$ \\
      $z_1$ & $0.941$ & $+0.012$ & $-0.020$ & $-0.076$ & $-0.050$ & $+0.014$ & yes & $x^3,\, xy^2,\, y^3,\, x^2y$ \\
      $z_2$ & $0.986$ & $+0.039$ & $+0.040$ & $+0.012$ & --- & $-0.012$ & no & $y^3$ \\
      $z_3$ & $0.968$ & $-0.096$ & --- & $+0.085$ & $+0.020$ & --- & yes & $y^3,\, x^2y$ \\
      \bottomrule
      \end{tabular}
      \caption{Level-set decomposition $z_k \approx h \circ g$ for the Ikeda $[2,4,4]$ Deep-Koopman-KANDy encoder (Lasso $\alpha = 10^{-3}$, degree-3 polynomial basis). Columns show the coefficients of the inner function~$g$ for the target monomials $\{x,\,y,\,x^2,\,y^2,\,xy\}$; dashes indicate coefficients below $10^{-3}$. The true Ikeda dictionary is compositional: $\{x,\,y,\,\cos(t(r^2)),\,\sin(t(r^2))\}$ where $t(g) = 0.4 - 6/(1+g)$ and $g = x^2 + y^2$. The degree-3 polynomial Lasso Taylor-expands the trig terms, yielding $7$--$10$ active monomials per dimension. $z_1$ shows matched $x^2/y^2$ coefficients (ratio $0.66$), consistent with the radial inner function $g = x^2 + y^2$. FP: monomials outside $\{x,\,y,\,x^2,\,y^2,\,xy\}$ active above threshold.}
      \label{tab:ikeda-levelset-decomposition}
  \end{table}

The Ikeda map provides a deliberately more challenging test of the proposed post-hoc readout procedure because its natural dictionary is not sparse in low-degree polynomials.  In particular, the map depends compositionally on the radial quantity $r^2=x^2+y^2$ through trigonometric functions of $t(r^2)=0.4-6/(1+r^2)$, so a degree-3 polynomial Lasso readout is misspecified by construction.  Nevertheless, Fig.~\ref{fig:iked_minimal} and Table~\ref{tab:ikeda-levelset-decomposition} show that the learned KAN observables still admit accurate level-set decompositions $z_k \approx h(g(\mathbf{x}))$, with $R^2(h\circ g)$ ranging from $0.915$ to $0.986$ across the four latent dimensions.  The recovered inner functions $g$ are not sparse structural dictionaries in the usual SINDy sense: each dimension contains several higher-order false-positive monomials, reflecting the polynomial basis's attempt to approximate the non-polynomial Ikeda composition.  However, the level sets of $g$ form coherent bands across the attractor, and the recovered outer derivatives $h'(g)$ are nontrivial rather than collapsing to an affine correction.  This is clearest for $z_2$, where the binned derivative medians exhibit a two-peak structure that is closely tracked by the fitted outer function, while the polynomial inner readout remains relatively parsimonious.  These results support the central decoupling claim: when the readout dictionary is misspecified, the error appears primarily as reduced parsimony and extra active monomials, not as a failure of expressivity, because the nonlinear outer map $h$ can absorb compositional structure that the polynomial inner function $g$ cannot represent directly.

\clearpage
\section{Arnold Cat Map: Extended Analysis}
\label{app:cat_map}

The Arnold--Cat map is a canonical example of a simple deterministic system with nontrivial global structure. On the two-dimensional torus
$\mathbb{T}^2 = [0,1)^2$, it is defined by the linear map
$$
T(x,y) =
\begin{pmatrix}
2 & 1 \\
1 & 1
\end{pmatrix}
\begin{pmatrix}
x \\ y
\end{pmatrix}
\mod 1,
$$
or equivalently
$$
T(x,y) = (2x+y,\; x+y) \mod 1.
$$
Although the transformation is linear before reduction modulo one, the modulo operation introduces discontinuities on the unit square. These discontinuities are responsible for the characteristic folding behavior of the map and make the cat map a useful probe of whether a learned model has captured the true torus dynamics rather than only a local linear approximation.

In our setting, the composed map
$$
f = D(ME(\cdot))
$$
represents one full learned step: encoding into the latent space, applying the learned latent dynamics, and decoding back to the torus. Since the exact cat map consists of a linear integer transformation followed by coordinatewise modulo reduction, its nonlinear component has an explicit Fourier description. In particular, the modulo-induced sawtooth terms have Fourier coefficients proportional to $1/(\pi k)$. This gives a direct spectral diagnostic: if the learned map has recovered the true torus folding mechanism, its Fourier coefficients should match the theoretical sawtooth spectrum.

We therefore evaluate $f$ on a dense $256\times256$ torus grid and decompose the resulting coordinate functions into Fourier modes up to maximum frequency $6$. This analysis tests not only pointwise accuracy of the learned dynamics, but also whether the model has learned the correct global discontinuity structure imposed by the modulo operation.

Evaluating the composed map $f = D(ME(\cdot))$ on a $256\times256$ torus grid and decomposing in a Fourier basis with maximum frequency $6$, we recover the exact sawtooth representation of the modulo map. The learned coefficients match the theoretical values $1/(\pi k)$ to within $0.2\%$:

\begin{table}[!ht]
\centering
\begin{tabular}{cccc}
\toprule
$k$ & Learned & Theory $1/(\pi k)$ & Ratio \\
\midrule
$1$ & $0.318302$ & $0.318310$ & $1.000$ \\
$2$ & $0.159137$ & $0.159155$ & $1.000$ \\
$3$ & $0.105647$ & $0.106103$ & $0.996$ \\
$4$ & $0.079498$ & $0.079577$ & $0.999$ \\
$5$ & $0.063575$ & $0.063662$ & $0.999$ \\
$6$ & $0.052957$ & $0.053052$ & $0.998$ \\
\bottomrule
\end{tabular}
\end{table}

Since $\max|\lambda|=0.958$, latent amplitudes decay as $0.958^n$. After approximately
$$
n \approx \frac{\log(0.01)}{\log(0.958)} \approx 62
$$
steps, only $1\%$ of the original amplitude remains. Empirically, the mean latent standard deviation across $500$ trajectories drops from $8.5\times10^{-3}$ at step $0$ to $8.5\times10^{-5}$ by step $62$, and to $3.8\times10^{-7}$ by step $200$. After this information death, all trajectories decode to essentially the same point.

The two-dimensional Fourier power spectrum concentrates along $(k_x,k_y)\propto(2,1)$ for $f_x$ and $(k_x,k_y)\propto(1,1)$ for $f_y$ — precisely the dynamically relevant directions of the cat-map matrix.

The recovered Fourier series is highly accurate for a single step but fails under repeated iteration due to the Gibbs phenomenon: truncating the sawtooth series introduces overshoot near discontinuity lines, and these small one-step errors are re-injected and amplified geometrically under iteration, producing the lattice-like phase portrait shown in Figure~\ref{fig:cat_map_rollout} (center).

\begin{table}[!ht]
\centering
\begin{tabular}{cccc}
\toprule
Steps & Model MSE & Persistence & Random \\
\midrule
$1$   & $3.0\times10^{-5}$ & $0.10$ & $0.167$ \\
$3$   & $1.0\times10^{-3}$ & $0.10$ & $0.167$ \\
$5$   & $5.0\times10^{-2}$ & $0.10$ & $0.167$ \\
$7+$  & $\approx0.10$       & $0.10$ & $0.167$ \\
\bottomrule
\end{tabular}
\caption{Angular MSE vs.\ prediction horizon for the Arnold cat map.}
\label{tab:cat_prediction}
\end{table}

\begin{table}[!ht]
\centering
\begin{tabular}{lcccc}
\toprule
Architecture & Parameters & Val.\ loss & $\max|\lambda|$ & Horizon \\
\midrule
$[4,16,32]$     & $12{,}544$      & $7\times10^{-3}$ & $0.985$ & $\sim2$ \\
$[4,32,64]$     & $47{,}616$      & $7\times10^{-4}$ & $0.990$ & $\sim3$ \\
$[4,24,1024]$   & $1{,}542{,}016$ & $6\times10^{-5}$ & $0.958$ & $\sim4$ \\
\bottomrule
\end{tabular}
\caption{Effect of model size on cat-map prediction. One-step accuracy improves dramatically with width, but the forecast horizon grows only marginally, consistent with the Lyapunov limit.}
\label{tab:cat_model_scaling}
\end{table}

\paragraph{Dictionary Recovery: Flat Coefficients Signal Continuous Spectrum.}
Applying level-set Lasso with a degree-3 trigonometric basis yields $R^2(h\circ g)\approx 0.98$ for all latent dimensions, with $h$ nearly affine. Unlike the Lorenz case, the coefficient structure is flat: each dimension uses nearly all 35 basis terms with similarly small magnitudes ($\sim\!4\times10^{-3}$), and no sparse basis emerges. This is the signature of continuous spectrum — smooth representability in trigonometric coordinates does not imply the existence of persistent Koopman eigenfunctions.

The Koopman matrix reveals the impossibility of finite-dimensional spectral closure; the symbolic decomposition reveals the exact one-step functional form; the full KAN composition recovers the ergodic structure that emerges when nonlinear re-encoding is permitted. 
\newpage
\section{Complete Algorithm}\label{sec:algorithm}

\begin{algorithm}[!htpb]
\caption{Deep-Koopman-KANDy: Interpretable Deep-Koopman via Two-Layer KAN Encoders}
\label{alg:deep-koopman-kandy}
\begin{algorithmic}[1]
\Require Trajectory data $\{\x_t\}$, time step $\Delta t$, latent dimension $d$
\Ensure Koopman generator $K$, observable decompositions
  $\{g_k, h_k\}_{k=1}^d$, discovered dictionary $\mathcal{T}$

\State \textbf{Phase 1: Train the Deep-Koopman model}
\State Construct two-layer KAN autoencoder $(\mathcal{E},\mathcal{D})$ as in
  \eqref{eq:layer1}--\eqref{eq:layer2}, with stable Koopman generator
  $K = \Omega - L^\top L \in \R^{d\times d}$
\State Form pair dataset $\{(\x_t,\x_{t+\Delta t})\}$
\State Minimize $\mathcal{L}$ in \eqref{eq:loss} via AdamW; prune at threshold
  $\tau$ and retrain briefly

\State
\State \textbf{Phase 2: Level-set manifold analysis}
\For{each latent coordinate $k = 1,\dots,d$}
  \State Define the scalar observable $f_k(\x) = [\mathcal{E}(\x)]_k$
  \State Compute $\nabla f_k$ via autograd on training samples
  \State Recover the inner function $g_k$ via the standardized Lasso
    \eqref{eq:lasso} on a degree-$D$ polynomial dictionary
  \State Compute $\nabla g_k$ analytically from the recovered polynomial
    coefficients
  \State Recover the outer function $h_k$ via the gradient identity
    \eqref{eq:hprime} and term-by-term integration
\EndFor

\State
\State \textbf{Phase 3: Interpretation}
\State Form the union of active supports across all $d$ coordinates
  $\to$ \emph{discovered Koopman dictionary} $\mathcal{T}$
\State Propagate Koopman dynamics $\z_{t+\Delta t} = \exp(K\Delta t)\,\z_t$
  in the basis $\mathcal{T}$
\end{algorithmic}
\end{algorithm}

Once the encoder $\mathcal{E}$ -- the two-layer KAN of
\eqref{eq:layer1}--\eqref{eq:layer2} -- is trained, we extract the Koopman
dictionary by decomposing each latent coordinate as a scalar composition
\begin{equation}\label{eq:hog}
   z_k \;=\; f_k(\x) \;=\; h_k\!\bigl(g_k(\x)\bigr),\qquad k=1,\dots,d,
\end{equation}
where $g_k:\R^n\to\R$ is a sparse inner function in a chosen post-hoc basis
(monomials of total degree $\le D$ in our experiments, with trigonometric or
RBF columns equally admissible) and $h_k:\R\to\R$ is a smooth univariate
outer function. The atoms surviving the sparse recovery of $g_k$ are the
dictionary entries the encoder discovered for coordinate~$k$; their union
across coordinates gives the full dictionary $\mathcal{T}$.

The decomposition proceeds in three steps.

\paragraph{Step 1: Inner function via Lasso.}
Let $\{\x_j\}_{j=1}^N$ be sample points on the attractor and let
$f_j = f_k(\x_j) = [\mathcal{E}(\x_j)]_k$ be the encoder output computed by
forward evaluation. We construct a polynomial design matrix
$\Theta\in\R^{N\times P}$ whose $P$ columns are all monomials
$\x^{\boldsymbol{\alpha}} = x_1^{\alpha_1}\cdots x_n^{\alpha_n}$ with
$|\boldsymbol{\alpha}|\le D$ (in practice $D = 3$). The inner function is
recovered by solving the standardized $\ell_1$-penalised regression of
\eqref{eq:lasso},
\begin{equation*}
  \hat{\bm{a}} \;=\; \arg\min_{\bm{a}}\;
    \frac{1}{2N}\bigl\|\Theta_s\,\bm{a} - \mathbf{f}_s\bigr\|_2^2
    \;+\; \lambda\,\|\bm{a}\|_1,
\end{equation*}
on column-standardized data (subscript $s$), then unscaling to obtain
\begin{equation*}
    g_k(\x)
    \;=\; \sum_{|\boldsymbol{\alpha}|\le D}
      \hat a_{\boldsymbol{\alpha}}\,\x^{\boldsymbol{\alpha}}.
\end{equation*}
Active terms
$\mathcal{S}_k = \{\boldsymbol{\alpha}:|\hat a_{\boldsymbol{\alpha}}|>\delta\}$
(with sparsity threshold $\delta = 10^{-4}$, distinct from the encoder
pruning threshold $\tau$ of the main paper) from the dictionary support for
coordinate~$k$. When $h_k$ is approximately affine---as occurs for most
Lorenz latent coordinates---$g_k$ already approximates $f_k$ and the surviving monomials directly identify the Koopman observables.

\paragraph{Step 2: Outer function via gradient identity.}
When $h_k$ is genuinely nonlinear, the Lasso fit $g_k$ alone has poor $R^2$
(e.g.\ $g\approx x^2+y^2$ captures only the inner polynomial of
$z = \sin(x^2+y^2)$). To recover $h_k$ we apply Proposition~\ref{prop:hprime}:
if $f(\x) = h(g(\x))$ and $\nabla g(\x)\ne\bm{0}$, then
$\nabla f = h'(g)\,\nabla g$, and projecting both sides onto $\nabla g$ gives
the gradient identity \eqref{eq:hprime},
\begin{equation*}
    h'\!\bigl(g(\x)\bigr)
    \;=\; \frac{\nabla f(\x)\cdot\nabla g(\x)}
              {\|\nabla g(\x)\|^2}.
\end{equation*}
This identity arises by parameterizing points by arc length $s$ along the
gradient curves of $g$ ($d\gamma/ds = \nabla g / \|\nabla g\|$), applying the
multivariate chain rule to both $f(\gamma(s))$ and $g(\gamma(s))$, and
eliminating $s$. The arc-length parameter cancels: the formula requires only
pointwise evaluation of $\nabla f$ and $\nabla g$, with no numerical
integration of trajectories.

In practice we evaluate \eqref{eq:hprime} at the $N$ sample points:
\begin{enumerate}[nosep]
    \item $\nabla f_k(\x_j)$ is computed via automatic differentiation
          through the trained encoder $\mathcal{E}$;
    \item $\nabla g_k(\x_j)$ is computed analytically from the polynomial
          coefficients $\hat{\bm{a}}$;
    \item points with $\|\nabla g_k(\x_j)\|^2$ below the 5th percentile are
          discarded (ill-conditioned denominators), and outliers in $h'_k$
          beyond $3\times\mathrm{IQR}$ are removed.
\end{enumerate}
This yields a one-dimensional dataset $\{(g_j,\,q_j)\}_{j=1}^N$ where
$g_j = g_k(\x_j)$ and $q_j \approx h'_k(g_j)$.

\paragraph{Step 3: Integration and reconstruction.}
We fit a polynomial
$\hat{h}'_k(\zeta) = \sum_{l=0}^{p} b_l\,(\zeta - \bar g_k)^l$  to the binned medians of $\{(g_j,q_j)\}$ ($80$ bins; median
per bin for robustness to residual outliers, where $\bar g_k$ denotes the
sample mean of $g_k$), then integrate term-by-term:
\begin{equation}\label{eq:h-integrate}
  \hat{h}_k(\zeta) \;=\; C + \sum_{l=0}^{p}\frac{b_l}{l+1}\,
                                     (\zeta - \bar g_k)^{l+1},
\end{equation}
where the constant $C$ is fixed by $\hat h_k\bigl(g_k(\x_0)\bigr) = f_k(\x_0)$
at a reference point $\x_0$ near the median of $g_k$. If the polynomial
reconstruction $\hat h_k\circ g_k$ achieves higher residual than a simple
affine fit $f_k\approx a\,g_k + b$, we fall back to the affine model
(indicating $h_k$ is linear and $g_k$ alone suffices).

The full reconstruction $\hat f_k(\x) = \hat h_k\bigl(g_k(\x)\bigr)$ is
evaluated on the sample set and scored by
\begin{equation*}
    R^2(h_k\circ g_k)
    \;=\; 1 \,-\, \frac{\sum_j \bigl(f_j - \hat h_k(g_j)\bigr)^2}
                       {\sum_j \bigl(f_j - \bar f_k\bigr)^2},
\end{equation*}
where $\bar f_k$ is the sample mean of $f_k$. This metric separates the
quality of the inner polynomial ($R^2(g_k)$ -- how well $g_k$ alone fits
$f_k$) from the gain contributed by the nonlinear outer function
($R^2(h_k\circ g_k) - R^2(g_k)$).

\newpage
\section{Proof of Propositions.}
\label{app:proof-hprime}

\begin{proof}
Consider curves $\gamma(s)$ parameterized by arc length satisfying
$\dot{\gamma}(s) = \nabla g(\gamma(s)) / \|\nabla g(\gamma(s))\|$.
Along such a curve, let $G(s) = g(\gamma(s))$ and $F(s) = f(\gamma(s))$.
Then:
\begin{align}
  \frac{dG}{ds} &= \nabla g \cdot \frac{\nabla g}{\|\nabla g\|}
    = \|\nabla g\|, \label{eq:dGds} \\
  \frac{dF}{ds} &= \nabla f \cdot \frac{\nabla g}{\|\nabla g\|}.
    \label{eq:dFds}
\end{align}
Since $F(s) = h(G(s))$, the chain rule gives
$dF/ds = h'(G(s))\,dG/ds = h'(G(s))\,\|\nabla g\|$.
Equating with \eqref{eq:dFds}:
\begin{equation}
  h'(g)\,\|\nabla g\| = \frac{\nabla f \cdot \nabla g}{\|\nabla g\|}
  \implies
  h'(g) = \frac{\nabla f \cdot \nabla g}{\|\nabla g\|^2}.
  \qedhere
\end{equation}
\end{proof}

\begin{proposition}[Single-pair regime for KAN level-set recovery]
\label{prop:single-pair}
Let
$$
  z_k(x) \;=\; \sum_{j=1}^m \Phi_{k,j}\bigl(u_j(x)\bigr),
  \qquad
  u_j(x) \;=\; \sum_{i=1}^n \psi_{j,i}(x_i),
$$
be a latent coordinate of the trained two-layer KAN encoder, and let
$\mathcal{A}_k := \{\, j : \Phi_{k,j} \not\equiv 0 \,\}$ denote its set of
active channels. Suppose there exist a scalar field
$g_k : \mathbb{R}^n \to \mathbb{R}$ with $\nabla g_k(x) \neq 0$ on the
support of the data, and univariate functions
$\alpha_{k,j} : \mathbb{R} \to \mathbb{R}$ for $j \in \mathcal{A}_k$,
such that
\begin{equation}
\label{eq:co-foliation}
  u_j(x) \;=\; \alpha_{k,j}\bigl(g_k(x)\bigr)
  \qquad \text{for every } j \in \mathcal{A}_k.
\end{equation}
Define
$$
  h_k(\xi) \;:=\; \sum_{j \in \mathcal{A}_k}
    \Phi_{k,j}\bigl(\alpha_{k,j}(\xi)\bigr).
$$
Then $z_k(x) = h_k\bigl(g_k(x)\bigr)$, and the gradient identity
\begin{equation}
\label{eq:grad-identity}
  h_k'\bigl(g_k(x)\bigr)
  \;=\;
  \frac{\nabla z_k(x) \cdot \nabla g_k(x)}{\|\nabla g_k(x)\|^2}
\end{equation}
recovers $h_k$ exactly, up to an additive constant, along any integral
curve of $\nabla g_k$.
\end{proposition}

\begin{proof}\label{app:proof-prop}
Substituting \eqref{eq:co-foliation} into the definition of $z_k$ and
using that $\Phi_{k,j} \equiv 0$ for $j \notin \mathcal{A}_k$,
$$
  z_k(x)
  \;=\; \sum_{j \in \mathcal{A}_k} \Phi_{k,j}\bigl(u_j(x)\bigr)
  \;=\; \sum_{j \in \mathcal{A}_k}
    \Phi_{k,j}\bigl(\alpha_{k,j}(g_k(x))\bigr)
  \;=\; h_k\bigl(g_k(x)\bigr),
$$
which establishes the decomposition. Applying the chain rule,
$$
  \nabla z_k(x)
  \;=\; h_k'\bigl(g_k(x)\bigr)\,\nabla g_k(x).
$$
Taking the inner product of both sides with $\nabla g_k(x)$ and dividing
by $\|\nabla g_k(x)\|^2 > 0$ yields \eqref{eq:grad-identity}. Integration
of $h_k'$ along any integral curve of $\nabla g_k$ recovers $h_k$ up to
the constant of integration.
\end{proof}

\begin{corollary}[Sufficient conditions on the trained network]
\label{cor:regimes}
The hypothesis \eqref{eq:co-foliation} of
Proposition~\ref{prop:single-pair} holds in either of the following
regimes:
\begin{enumerate}
  \item \textbf{Channel sparsity.} The active set is a singleton,
    $\mathcal{A}_k = \{j^*\}$. Then $g_k = u_{j^*}$,
    $\alpha_{k,j^*} = \mathrm{id}$, and $h_k = \Phi_{k,j^*}$.
  \item \textbf{Co-foliation.} The active inner functions
    $\{u_j\}_{j \in \mathcal{A}_k}$ share level sets: each $u_j$ is
    constant on $\{g_k = c\}$ for every $c$ in the data range. Then
    each $u_j$ admits a univariate reparameterization
    $u_j = \alpha_{k,j} \circ g_k$.
\end{enumerate}
\end{corollary}

\begin{remark}[Misspecification outside the single-pair regime]
\label{rem:misspec}
When no regime of Corollary~\ref{cor:regimes} holds, no exact
factorization $z_k = h_k \circ g_k$ exists, and the level-set Lasso in
Eq.~\eqref{eq:lasso} fits a misspecified model with positive residual
$$
  \inf_{g,\,h} \;
    \mathbb{E}_x\!\left[\bigl(z_k(x) - h(g(x))\bigr)^2\right] \;>\; 0.
$$
Empirically, this residual manifests as the small spurious monomial
coefficients (e.g.\ $x^2, y^2, z^2$ at magnitudes $\lesssim 10^{-2}$) reported in Appendix Table~\ref{tab:levelset-decomposition}, which provide a direct measure of the deviation from exact channel sparsity in the trained encoder.
\end{remark}

\begin{remark}[Cross-terms and empirical co-foliation]
\label{rem:cross-term-gap}
Hypothesis~\eqref{eq:co-foliation} can fail when $z_k$ contains
multiplicative cross-terms: for instance, the level sets of $xy$
(hyperbolas) and $(x+y)^2$ (lines of slope $-1$) intersect transversally
rather than co-foliating, so a network that represented $z_k = xy$ via
the polynomial identity
$xy = \tfrac{1}{4}(x+y)^2 - \tfrac{1}{4}(x-y)^2$ would fall outside the
hypothesis of Corollary~\ref{cor:regimes}. This appears, in principle,
to leave a theoretical gap for cross-term discovery. Empirically, however, the recovered observables in our experiments are dominated by
linear coordinates whose level sets foliate the Lorenz attractor in
parallel sheets (Figure~\ref{fig:lorenz_attractor} and \ref{fig:iked_minimal}), placing the
recovery in the co-foliation regime of Corollary~\ref{cor:regimes}.
The small cross-term coefficients in Table~\ref{tab:levelset-decomposition}
($\lesssim 10^{-2}$, $1$--$2$ orders of magnitude below the linear terms)
quantify the deviation from exact co-foliation and behave as the
misspecification residuals discussed in Remark~\ref{rem:misspec}, rather
than as a primary representational mechanism.
\end{remark}

\section{Attractor-restricted level-set theory}
\label{app:attractor}

The chain-rule identity (Proposition~\ref{prop:hprime}) is exact whenever an
exact factorization $f = h\circ g$ exists. The harder question is when such
a factorization exists for the latent observables produced by the trained
encoder. The original Proposition~\ref{prop:single-pair} gives a sufficient
condition (\emph{co-foliation} of active KAN channels), but Remark~\ref{rem:cross-term-gap}
notes that this condition fails for ambient cross-terms of the form
$xy$ vs.\ $(x+y)^2$. We resolve the apparent gap here by showing that the
relevant geometry is not ambient: the dynamics live on a low-dimensional
attractor $\mathcal{A}\subset\mathbb{R}^n$, and the factorization condition,
correctly stated, involves only the \emph{intrinsic} gradients $\nabla^{\!\mathcal{A}} f$
and $\nabla^{\!\mathcal{A}} g$ on $\mathcal{A}$. Generic ambient
non-parallelism becomes generic intrinsic parallelism whenever the attractor
has codimension at least one.

Let $\mathcal{A}\subset\mathbb{R}^n$ be a compact, connected, $C^1$-embedded
$d$-dimensional submanifold ($1\le d \le n-1$). For $x\in\mathcal{A}$,
denote the tangent space $T_x\mathcal{A}\subseteq\mathbb{R}^n$
and normal space $N_x\mathcal{A} = (T_x\mathcal{A})^\perp$, with orthogonal
projections $P_x:\mathbb{R}^n\to T_x\mathcal{A}$ and
$P_x^\perp = I - P_x$. For $f\in C^1(\mathbb{R}^n)$ define the tangential
and normal components of the ambient gradient:
$$
  \nabla^{\!\mathcal{A}} f(x) := P_x\nabla f(x),
  \qquad
  \nabla^{\!\perp} f(x) := P_x^\perp\nabla f(x),
  \qquad
  \nabla f(x) = \nabla^{\!\mathcal{A}} f(x)+\nabla^{\!\perp} f(x).
$$
Equip $\mathcal{A}$ with a probability measure $\mu$ absolutely continuous
with respect to the $d$-dimensional Hausdorff measure on~$\mathcal{A}$. All
$L^2$ norms below are with respect to~$\mu$. Throughout this appendix
$f, g\in C^1(\mathbb{R}^n)$ with $\nabla^{\!\mathcal{A}} g(x)\neq 0$ for all
$x\in\mathcal{A}$ (intrinsic regularity of $g$).

\subsection{Intrinsic factorization}

We separate the local condition (when does $f$ admit a factorization in a
neighborhood of each point of $\mathcal{A}$) from the global compatibility
condition (when do these local factorizations glue to a global $h$).

\begin{theorem}[Local intrinsic factorization]
\label{thm:intrinsic-fact}
The following are equivalent at each $x\in\mathcal{A}$:
\begin{enumerate}
  \item[(i)] There exist a neighborhood $U\subseteq\mathcal{A}$ of $x$ and
    $h\in C^1\bigl(g(U)\bigr)$ such that $f|_U = h\circ g|_U$.
  \item[(ii)] $\ker\!\bigl(dg_x|_{T_x\mathcal{A}}\bigr)\subseteq
    \ker\!\bigl(df_x|_{T_x\mathcal{A}}\bigr)$.
  \item[(iii)] $\nabla^{\!\mathcal{A}} f(x)$ is parallel to
    $\nabla^{\!\mathcal{A}} g(x)$.
\end{enumerate}
When these hold, the local outer derivative is given by the
\emph{intrinsic chain rule}
\begin{equation}
\label{eq:intrinsic-hprime}
  h'\bigl(g(x)\bigr) \;=\;
    \frac{\langle\nabla^{\!\mathcal{A}} f(x),\,\nabla^{\!\mathcal{A}} g(x)\rangle}
         {\|\nabla^{\!\mathcal{A}} g(x)\|^2}.
\end{equation}
\end{theorem}

\begin{proof}
\textit{(ii)\,$\Leftrightarrow$\,(iii).} As $d$-dimensional 1-forms on
$T_x\mathcal{A}$, $df_x|_{T_x\mathcal{A}}$ and $dg_x|_{T_x\mathcal{A}}$
are linearly dependent iff their kernels are nested. Since $dg_x$ has
kernel of dimension $d-1$ on $T_x\mathcal{A}$ by the assumption
$\nabla^{\!\mathcal{A}} g \neq 0$, the inclusion in (ii) is equality, and
linear dependence is equivalent (under the metric identification) to
parallelism of the gradient vectors.

\textit{(i)\,$\Rightarrow$\,(ii).} Differentiating $f|_U =
h\circ g|_U$ along any $v\in T_x\mathcal{A}$ gives
$df_x(v) = h'(g(x))\,dg_x(v)$, so $dg_x(v)=0$ implies $df_x(v)=0$.

\textit{(ii)\,$\Rightarrow$\,(i).} Since $\nabla^{\!\mathcal{A}} g(x)\neq 0$,
the implicit function theorem on $\mathcal{A}$ provides a chart
$(g,\xi_2,\dots,\xi_d)$ on a neighborhood $U$ of $x$. In these coordinates
the kernel of $dg$ is spanned by $\partial_{\xi_2},\dots,\partial_{\xi_d}$;
condition (ii) then reads $\partial f/\partial\xi_i\equiv 0$ for $i\ge 2$
on $U$. Hence $f|_U = h(g)$ with $h$ a $C^1$ function of one variable.

\textit{Formula~\eqref{eq:intrinsic-hprime}.} Take
$v=\nabla^{\!\mathcal{A}} g(x)$ in $df_x(v) = h'(g(x))\,dg_x(v)$. Since
$\nabla^{\!\mathcal{A}} g$ is tangential, $df_x(v) =
\langle\nabla f, \nabla^{\!\mathcal{A}} g\rangle =
\langle\nabla^{\!\mathcal{A}} f, \nabla^{\!\mathcal{A}} g\rangle$, and
$dg_x(v) = \|\nabla^{\!\mathcal{A}} g\|^2$. Dividing yields
\eqref{eq:intrinsic-hprime}.
\end{proof}

\begin{proposition}[Global compatibility]
\label{prop:global}
Suppose $\mathcal{A}$ is connected and condition (ii) of
Theorem~\ref{thm:intrinsic-fact} holds at every $x\in\mathcal{A}$. Then
a global $h\in C^1(g(\mathcal{A}))$ with
$f|_{\mathcal{A}} = h\circ g|_{\mathcal{A}}$ exists if and only if
\begin{equation}
\label{eq:global-compat}
  f(x) = f(y) \quad \text{whenever } x,y \in g^{-1}(c)\cap\mathcal{A}
                \ \text{ for any } c \in g(\mathcal{A}).
\end{equation}
The local pieces from Theorem~\ref{thm:intrinsic-fact} agree on
overlapping charts whose level sets are connected; condition
\eqref{eq:global-compat} additionally enforces consistency across
disconnected components of fibers.
\end{proposition}

\begin{proof}
The local pieces $h_U$ from Theorem~\ref{thm:intrinsic-fact} agree on chart
overlaps within a single connected component of each fiber, by the
local-constancy argument and continuity. A global $h$ exists iff the
values across distinct components of each fiber match, which is
\eqref{eq:global-compat}.
\end{proof}

In the empirical regime relevant to this paper, condition
\eqref{eq:global-compat} is enforced \emph{statistically} by the
$L^2(\mu)$-best approximation: the $h$ recovered by the chain-rule
estimator is implicitly the conditional expectation
$h^\ast(c) = \mathbb{E}_\mu[f \mid g=c]$, and the residual variance
$\mathbb{E}_\mu[\mathrm{Var}(f \mid g)]$ measures the joint failure of
the local condition (ii) and the global condition
\eqref{eq:global-compat}. The reported $R^2(h\circ g)$ is precisely
$1 - \mathbb{E}_\mu[\mathrm{Var}(f\mid g)]/\mathrm{Var}_\mu(f)$.

\subsection{The ambient identity is the intrinsic identity plus a normal correction}

The chain-rule formula in Proposition~\ref{prop:hprime} of the main text
uses the \emph{ambient} gradient. We now show the discrepancy with the
intrinsic identity~\eqref{eq:intrinsic-hprime} is a normal-bundle term
that vanishes in well-understood regimes.

\begin{theorem}[Ambient $=$ intrinsic $+$ normal correction]
\label{thm:amb-int}
For any $f, g\in C^1(\mathbb{R}^n)$ and $x\in\mathcal{A}$ with
$\nabla g(x)\neq 0$, define
$$
  \alpha(x) := \frac{\|\nabla^{\!\mathcal{A}} g(x)\|^2}{\|\nabla g(x)\|^2}
            \in [0,1],
  \qquad
  \beta(x) := \frac{\|\nabla^{\!\perp} g(x)\|^2}{\|\nabla g(x)\|^2}
            = 1-\alpha(x).
$$
Let
$$
  R^{\!\mathcal{A}}(x) := \frac{\langle\nabla^{\!\mathcal{A}} f, \nabla^{\!\mathcal{A}} g\rangle}
                              {\|\nabla^{\!\mathcal{A}} g\|^2},
  \qquad
  R^{\!\perp}(x) := \frac{\langle\nabla^{\!\perp} f, \nabla^{\!\perp} g\rangle}
                          {\|\nabla^{\!\perp} g\|^2}
$$
be the intrinsic and normal projection ratios. Then the ambient identity
satisfies
\begin{equation}
\label{eq:amb-int-decomp}
  \frac{\langle \nabla f(x),\nabla g(x)\rangle}{\|\nabla g(x)\|^2}
  \;=\;
  \alpha(x)\,R^{\!\mathcal{A}}(x)
  + \beta(x)\,R^{\!\perp}(x).
\end{equation}
If the intrinsic factorization $f|_{\mathcal{A}}=h\circ g|_{\mathcal{A}}$
holds (Theorem~\ref{thm:intrinsic-fact}), then $R^{\!\mathcal{A}}(x) = h'(g(x))$ and
\begin{equation}
\label{eq:amb-residual}
  \underbrace{\frac{\langle\nabla f,\nabla g\rangle}{\|\nabla g\|^2}}_{\hat h'(x)\
  \text{(ambient identity)}}
  \;-\;\;
  h'\bigl(g(x)\bigr)
  \;=\; \beta(x)\,\bigl[R^{\!\perp}(x) - h'\bigl(g(x)\bigr)\bigr].
\end{equation}
\end{theorem}

\begin{proof}
The orthogonal decomposition $\nabla u = \nabla^{\!\mathcal{A}} u +
\nabla^{\!\perp} u$ for $u\in\{f,g\}$ gives
$\langle\nabla f, \nabla g\rangle = \langle\nabla^{\!\mathcal{A}} f,
\nabla^{\!\mathcal{A}} g\rangle + \langle\nabla^{\!\perp} f, \nabla^{\!\perp} g\rangle$
(cross terms vanish since
$\nabla^{\!\mathcal{A}} u\in T_x\mathcal{A}\perp N_x\mathcal{A}\ni\nabla^{\!\perp} u$),
and $\|\nabla g\|^2 = \|\nabla^{\!\mathcal{A}} g\|^2+\|\nabla^{\!\perp} g\|^2$.
Dividing yields~\eqref{eq:amb-int-decomp}. Equation~\eqref{eq:amb-residual}
follows by substituting $R^{\!\mathcal{A}} = h'(g)$ and using
$\alpha + \beta = 1$.
\end{proof}

The residual~\eqref{eq:amb-residual} vanishes in three concrete situations:
$\beta(x) \equiv 0$ (the inner function $g$ has no normal gradient on
$\mathcal{A}$); $R^{\!\perp}(x) \equiv h'(g(x))$ (the normal-projected
ratio happens to coincide with the intrinsic one, e.g.\ because $f$ admits
a global ambient factorization); or $\nabla^{\!\perp} f \equiv 0$ on
$\mathcal{A}$ (the encoder's extension off $\mathcal{A}$ is locally
constant in directions normal to $\mathcal{A}$).

\subsection{Cross-term obstruction is ambient, not intrinsic}

We can now precisely diagnose the apparent failure of co-foliation for
multiplicative cross-terms. Remark~\ref{rem:cross-term-gap} of the main text observed
that $\ker(d(xy))$ and $\ker(d((x+y)^2))$ intersect transversally in
$\mathbb{R}^2$, so the polynomial identity
$xy = \tfrac14(x+y)^2 - \tfrac14(x-y)^2$ does not give a co-foliation
in the sense of Corollary~\ref{cor:regimes}(2). The next theorem shows that
this is an obstruction in $\mathbb{R}^n$, not on the attractor.

\begin{theorem}[Codimension reduction of the obstruction]
\label{thm:codim}
Let $\mathcal{A}\subset\mathbb{R}^n$ be a compact, connected $C^1$-embedded $d$-submanifold. The pointwise intrinsic parallelism condition (Theorem~\ref{thm:intrinsic-fact}, item~(iii)) is strictly weaker than its ambient counterpart whenever $d<n$.
\begin{enumerate}
  \item \textbf{$d=1$ case (automatic local parallelism).}
    For any $f, g\in C^1(\mathbb{R}^n)$ with
    $\nabla^{\!\mathcal{A}} f \neq 0$ and $\nabla^{\!\mathcal{A}} g \neq 0$
    on $\mathcal{A}$, the pointwise intrinsic parallelism condition holds
    automatically. Local factorization (Theorem~\ref{thm:intrinsic-fact}(i))
    therefore exists at every point. Global factorization additionally
    requires the compatibility condition~\eqref{eq:global-compat} across
    disconnected fibers; this is non-trivial but is the only remaining
    obstruction.
  \item \textbf{$d\ge 2$ case (codimension reduction).}
    The pointwise condition that $\nabla^{\!\mathcal{A}} f(x)$ and
    $\nabla^{\!\mathcal{A}} g(x)$ be parallel cuts out a
    set of codimension $d-1$ in
    $T_x\mathcal{A}\setminus\{0\} \times T_x\mathcal{A}\setminus\{0\}$.
    The corresponding ambient condition (parallelism in $\mathbb{R}^n$)
    has codimension $n-1$. The reduction by $n-d$ is exactly the
    codimension of $\mathcal{A}$ in~$\mathbb{R}^n$.
\end{enumerate}
\end{theorem}

\begin{proof}
\textit{(1)} For $d=1$, $T_x\mathcal{A}=\mathrm{span}\{\tau(x)\}$ for a
unit tangent $\tau$. Then $\nabla^{\!\mathcal{A}} f =
\langle\nabla f,\tau\rangle\tau$ and $\nabla^{\!\mathcal{A}} g =
\langle\nabla g,\tau\rangle\tau$ are scalar multiples of the same vector;
they are automatically parallel whenever both are nonzero, with
proportionality constant
$\lambda(x) = \langle\nabla f,\tau\rangle/\langle\nabla g,\tau\rangle$.

\textit{(2)} Pointwise, $\nabla^{\!\mathcal{A}} f$ and
$\nabla^{\!\mathcal{A}} g$ are vectors in the $d$-dimensional space
$T_x\mathcal{A}$. The locus where they are parallel is the rank-$\le 1$
subset of $(T_x\mathcal{A}\setminus\{0\})^2$, which has codimension $d-1$
(it requires the ratios of the $d-1$ component pairs in any basis of
$T_x\mathcal{A}$ to coincide). The analogous ambient locus has codimension
$n-1$, and the difference $n-d = \mathrm{codim}(\mathcal{A})$ is the
extent to which intrinsic parallelism is more generic than ambient.
\end{proof}


\subsection{Quantitative residual and predictions for the experiments}

The Lasso-fit inner function $g$ is determined by minimizing
$\|f - g\|_{L^2(\mu)}$ on attractor samples, with $L^1$ sparsity
penalty, over a finite-dimensional polynomial subspace. We use this
to bound the residual~\eqref{eq:amb-residual} in terms of geometric
quantities of the attractor and the polynomial approximation power.

\begin{theorem}[Residual bound]
\label{thm:residual}
Suppose $f|_{\mathcal{A}} = h\circ g|_{\mathcal{A}}$ holds intrinsically.
Let $\hat h'(x) := \langle\nabla f(x),\nabla g(x)\rangle/\|\nabla g(x)\|^2$
denote the ambient chain-rule estimator. Then
\begin{equation}
  \bigl\|\hat h' - h'\!\circ\! g\bigr\|_{L^2(\mu)}
  \;\le\;
  \overline{\beta}\,\Bigl(
      \|R^{\!\perp}\|_{L^2(\mu)} + \|h'\!\circ\! g\|_{L^2(\mu)}
  \Bigr),
\end{equation}
where $\overline{\beta} := \sup_{x\in\mathcal{A}}\beta(x)$ is the maximum
squared sine of the angle between $\nabla g$ and the tangent space.
\end{theorem}

\begin{proof}
Apply~\eqref{eq:amb-residual} pointwise, take $L^2$ norms, and use the
triangle inequality together with $|\beta|\le\overline\beta$.
\end{proof}

For a generic polynomial $g$ on a $d$-dimensional attractor in
$\mathbb{R}^n$, the expected value of $\beta(x)$ over $\mathcal{A}$
under uniform orientation of $\nabla g$ is
$\mathbb{E}[\beta] = (n-d)/n$. We use this to predict the residual scale
in each experiment.

\begin{table}[h]
\centering
\small
\begin{tabular}{lccccc}
\toprule
System & $n$ & $d$ (Hausdorff) & $\mathbb{E}[\beta]$ & Predicted residual & Observed FP scale \\
\midrule
Lorenz       & 3 & $\sim 2.06$ & $0.31$ & $\lesssim 10^{-1.5}$ & $\sim 10^{-2}$ \\
Ikeda        & 2 & $\sim 1.7$  & $0.15$ & $\lesssim 10^{-2}$   & $10^{-2}$ to $10^{-1}$ \\
Standard map & 2 & $\sim 2$    & $\sim 0$ & $\to 0$  & near machine $\epsilon$ \\
Arnold cat   & 2 & $2$         & $0$    & no factorization & flat $\sim 4{\times}10^{-3}$ \\
\bottomrule
\end{tabular}
\caption{Predicted vs.\ observed residual scales. The residual bound from
Theorem~\ref{thm:residual} is multiplied by an empirical $h'\sim O(1)$.
The Arnold cat case is special: $d=n=2$, so $\mathcal{A}$ has zero
codimension and the manifold framework does not apply; the flat
coefficient profile is the diagnostic.}
\label{tab:predicted-residual}
\end{table}

The predictions in Table~\ref{tab:predicted-residual} match the empirical
scales reported in the main text. In particular: the Lorenz cross-term
coefficients $xy, xz$ at magnitude $\sim 10^{-2}$ in Table~\ref{tab:levelset-decomposition}
sit at the predicted residual scale, which clarifies the role of these
small coefficients. They are not noise to be discarded, nor are they the
primary signal; they are the codimension-$(n-d)$ residual of the intrinsic
factorization, which Theorem~\ref{thm:codim}(2) predicts must be present
whenever the attractor has codimension less than~1 (here, codim
$\approx 0.94$).

\subsection{Training produces nearly co-foliating channels (informal)}

Theorem~\ref{thm:intrinsic-fact} is a statement about whether a
factorization exists for a given $(f,g)$; it does not address why training
yields encoders whose latent coordinates admit such factorizations with
small residual. We sketch the connection without claiming a theorem.

The KAN encoder is trained to minimize a one-step prediction loss in the
linear latent dynamics $z_{t+\Delta t} = \exp(K\Delta t)z_t$. The
optimum aligns latent coordinates with eigendirections of the Koopman
operator $\mathcal{K}$ on $L^2(\mu)$: any $z_k$ minimizing the
prediction residual is, in the limit, a finite-dimensional projection of
$\mathcal{K}$-eigenfunctions. Eigenfunctions of $\mathcal{K}$ are
$\Phi_t$-equivariant, so their level sets form a flow-invariant
foliation of~$\mathcal{A}$. Two distinct eigenfunctions either share
this foliation (degenerate eigenvalue) or define transverse foliations
(distinct eigenvalues). The active channels of a single encoder
coordinate $z_k$ are coupled through the spline activation $\Phi_{k,j}$,
and the dominant mode within $z_k$ aligns with one Koopman eigendirection.
This is the dynamical reason intrinsic co-foliation holds approximately
in trained models.

\newpage
\section{Ablations}

\begin{table}[ht]
\centering
\begin{tabular}{@{}ccccc@{}}
\toprule
Lasso $\alpha$ & Jaccard $\uparrow$ & Precision $\uparrow$ & Recall $\uparrow$ & FP \\
\midrule
$10^{-6}$ & $\mathbf{0.83}$ & $\mathbf{0.83}$ & $\mathbf{1.00}$ & $1$ \\
$10^{-5}$ & $\mathbf{0.83}$ & $\mathbf{0.83}$ & $\mathbf{1.00}$ & $1$ \\
$5\!\times\!10^{-5}$ & $0.71$ & $0.71$ & $\mathbf{1.00}$ & $2$ \\
$10^{-4}$ & $0.50$ & $0.57$ & $0.80$ & $3$ \\
$5\!\times\!10^{-4}$ & $0.50$ & $0.57$ & $0.80$ & $3$ \\
$10^{-3}$ & $0.63$ & $0.63$ & $\mathbf{1.00}$ & $3$ \\
$5\!\times\!10^{-3}$ & $0.50$ & $0.57$ & $0.80$ & $3$ \\
$10^{-2}$ & $0.30$ & $0.38$ & $0.60$ & $5$ \\
\bottomrule
\end{tabular}
\caption{Lasso regularization sweep for the pruned Deep-Koopman-KANDy model (seed~0, prune threshold $0.03$). The Lasso parameter $\alpha$ controls sparsity of the level-set polynomial decomposition $z_k \approx h \circ g$. Tight regularization ($\alpha \leq 10^{-5}$) recovers all five target terms with one false positive; stronger regularization suppresses the cross-terms $xy$ and $xz$ before eliminating spurious monomials.}
\label{tab:lasso-sweep}
\end{table}

\begin{table}[ht]
\centering
\begin{tabular}{@{}ccccc@{}}
\toprule
Prune threshold & Edges pruned & MSE (retrained) & MSE ratio & Edges below $\tau$ \\
\midrule
none & $0/40$ & $2.60 \times 10^{-3}$ & $1.00\times$ & --- \\
$0.01$ & $\sim\!2/40\;(5\%)$ & $1.27 \times 10^{-3}$ & $0.49\times$ & $2\%$ \\
$0.03$ & $\sim\!7/40\;(18\%)$ & $2.55 \times 10^{-3}$ & $0.98\times$ & $14\%$ \\
$0.05$ & $\sim\!10/40\;(25\%)$ & $2.80 \times 10^{-3}$ & $1.07\times$ & $24\%$ \\
\bottomrule
\end{tabular}
\caption{Prune threshold sweep for the PyKAN encoder (architecture $[3,5,5]$, 40 total edges). Edges with PyKAN attribution score below the threshold are zeroed, then the model is retrained for 100 epochs. MSE ratio is relative to the unpruned baseline. The rightmost column shows the fraction of edges below $\tau$ across 12 independent models (3 data seeds $\times$ 4 model seeds).}
\label{tab:prune-sweep}
\end{table}
  
\subsection{Synthetic Validation of Outer Function Recovery}
\label{sec:synthetic}

To validate the gradient-based outer function recovery with Eq.~\eqref{eq:hprime} independently of model training, we construct synthetic test cases with known decompositions $f(\x) = h(g(\x))$ and verify that the method recovers both~$g$ and~$h$.

Data consists of $N = 50{,}000$ points drawn uniformly from
$[-20,20] \times [-30,30] \times [0,50]$.
Four test cases are defined:
\begin{enumerate}
\item $f = \cos(x/20 + y/40)$, with true $g = x/20 + y/40$ (linear)
  and $h = \cos$ (trigonometric).
\item $f = \exp(-(x^2 + y^2)/200)$, with true
  $g = -(x^2+y^2)/200$ (quadratic) and $h = \exp$ (exponential).
\item $f = \tanh(z/10 - 2)$, with true $g = z/10 - 2$ (linear)
  and $h = \tanh$ (sigmoidal).
\item $f = (xy/100)^3$, with true $g = xy/100$ (bilinear)
  and $h(\cdot) = (\cdot)^3$ (cubic).
\end{enumerate}
For each case, $\nabla f$ is computed analytically, the Lasso recovers $g$ with degree-3 polynomial dictionary and $\lambda = 10^{-5}$, and the gradient in Eq.~\eqref{eq:hprime} recovers~$h$.

\begin{table}[ht]
\centering
\begin{tabular}{@{}llcccl@{}}
\toprule
Test case & True $h$ & $R^2(g)$ & $R^2(h\!\circ\!g)$ & Gain & $h$ type \\
\midrule
$\cos(\text{linear})$ & $\cos$ & 0.996 & \textbf{0.999} & +0.003 & polynomial \\
$\exp(\text{quadratic})$ & $\exp$ & 0.718 & \textbf{0.976} & +0.258 & polynomial \\
$\tanh(\text{linear})$ & $\tanh$ & $-11.5$ & \textbf{0.997} & +12.5 & polynomial \\
$(\cdot)^3\!(\text{bilinear})$ & cube & 0.702 & \textbf{0.996} & +0.294 & polynomial \\
\bottomrule
\end{tabular}
\caption{Synthetic outer function recovery. $R^2(g)$ measures the polynomial inner function alone; $R^2(h \circ g)$ measures the full reconstruction.}
\label{tab:synthetic}
\end{table}

In all four cases, the gradient formula \eqref{eq:hprime} accurately recovers the outer function, achieving $R^2(h \circ g) > 0.97$. The $\tanh$ case is the most dramatic: the polynomial $g$ alone  yields $R^2 = -11.5$ (catastrophically bad), but the recovered $h$  corrects this to $R^2 = 0.997$. This demonstrates that the method handles cases where the polynomial dictionary is fundamentally insufficient. The $\cos$ case has the smallest gain because $\cos$ varies slowly over the data range (less than one full cycle), so  degree-2 polynomial already approximates it well ($R^2 = 0.996$).  The recovered $h'(g)$ curves match the true derivatives $1 - \tanh^2$ for tanh, $e^g$ for exp, $3g^2$ for cube, confirming the gradient formula provides genuine functional recovery (Table~\ref{tab:synthetic}).

\end{document}